\documentclass{amsart}

\usepackage{amssymb}
\usepackage{latexsym}
\usepackage{amsmath}
\usepackage{enumerate}
\usepackage{amsthm}
\usepackage{wasysym}
\usepackage{stmaryrd}
\usepackage{mathrsfs}
\usepackage{pifont}
\usepackage{tikz}
\usepackage[f]{esvect}

\newcommand{\qee} {\hspace*{2mm}\hfill \ding{109}}

\renewcommand{\iff}{\leftrightarrow}

\renewcommand{\phi}{\varphi}

\renewcommand{\Theta}{\varTheta}
\renewcommand{\Phi}{\varPhi}
\renewcommand{\Psi}{\varPsi}
\renewcommand{\Xi}{\varXi}
\renewcommand{\Omega}{\varOmega}

\newcommand{\qedright}{\belowdisplayskip=-12pt}

\newtheorem{theorem}{Theorem}[section]
\newtheorem{define}[theorem]{Definition}

\newtheorem{exa}[theorem]{Example}
\newenvironment{example}{\begin{exa} \rm}{\qee\end{exa}}
\newtheorem{exerc}[theorem]{Exercise}

\newtheorem{conj}[theorem]{Conjecture}

\newtheorem{ques}[theorem]{Open Question}
\newenvironment{question}{\begin{ques} \rm}{\qee\end{ques}}
\newtheorem{lemma}{Lemma}[section]
\newtheorem{corollary}{Corollary}[section]
\newtheorem{rem}[theorem]{Remark}
\newenvironment{remark}{\begin{rem} \rm}{\qee\end{rem}}

\DeclareMathOperator{\necessary}{\text{\tikz[scale=.6ex/1cm,baseline=-.6ex,line width=.1ex]{
                            \draw (-1,-1) rectangle (1,1);}}}

  \DeclareMathOperator{\xpossible}{\text{\tikz[scale=.6ex/1cm,baseline=-.6ex,rotate=45,line width=.1ex]{
                            \draw (-1,-1) rectangle (1,1); \draw[very thin] (-.6,-.6) rectangle (.6,.6);}}}

                            \newcommand{\tightbox}[1]{%
  \tikz{\node[rectangle,draw,inner sep=0.25pt,line width=0.25pt] (A) {$#1$};}%
}
\newcommand{\boxvee}{\mathbin{\text{\small \tightbox{\vee}}}}
                                                        
\newcommand{\mc}[1]{\mathcal {#1}}
\newcommand{\mf}[1]{\mathfrak {#1}}
\newcommand{\medent}{\medskip\noindent}
 \newcommand{\tupel}[1]{{\langle #1 \rangle}}
\newcommand{\verz}[1]{\{ #1 \}}
\newcommand{\To}{\Rightarrow}
\newcommand{\Iff}{\Leftrightarrow}

\newcommand{\apr}{{\vartriangle}}

\newcommand{\opr}{\necessary}

\newcommand{\jump}{\mathrel{\mbox{\textcolor{gray}{$\blacktriangleright$}}}}
\newcommand{\jumpb}{\mathrel{\mbox{\textcolor{gray}{$\blacktriangleleft$}}}}

 \newcommand{\nrhd}{\mathrel{\not\! \rhd}}

\newcommand{\mutint}{\bowtie}
\newcommand{\sls}{\subseteq_{\sf e}}
\newcommand{\sle}{\supseteq_{\sf e}}
\newcommand{\downp}[1]{#1_{\mf p}}
\newcommand{\downr}[1]{#1_{\mf r}}

   \definecolor{grre}{rgb}{0.7,0.5,0.5}

\newcommand{\sline}{\raise-0.3ex\hbox{$\hbox{--}\kern-0.84ex\raise0.45ex\hbox{$\hbox{\scalebox{0.3}{\bf /}}
\kern-0.37ex\hbox{\scalebox{0.3}{\bf /}}$}$}}
\newcommand{\slinei}{\raise-0.3ex\hbox{$\hbox{--}\kern-0.84ex\raise0.45ex\hbox{$\hbox{\scalebox{0.3}{\bf \textbackslash}}
\kern-0.37ex\hbox{\scalebox{0.3}{\bf \textbackslash}}$}$}}
\newcommand{\lhdnneq}{\mathrel{\lhd_{\hspace*{-0.27cm}{}_{\kern0.2ex \slinei}}\hspace*{0.09cm}}}
\newcommand{\rhdnneq}{\mathrel{\rhd_{\hspace*{-0.27cm}{}_{\kern0.3ex \sline}}\hspace*{0.09cm}}}
\newcommand{\sva}{\phi}
\newcommand{\svb}{\psi}

\newcommand{\sigv}{\sigma}
\newcommand{\idt}[1]{\mathfrak{id}_{#1}}
\newcommand{\idtb}[2]{\mathfrak{id}_{#1}^{#2}}
\newcommand{\ccz}{\ensuremath{\xpossible}}

    \newcommand{\braee}[1]{\lceil #1 \rceil}
        \newcommand{\card}{{\mathfrak c}}
            \newcommand{\imod}[1]{{\hat #1}}
            \newcommand{\tol}{\Rsh}
             \newcommand{\ehu}{essentially hereditarily undecidable}
             \newcommand{\spth}[1]{{\mathfrak #1}}
             \newcommand{\pams}{{{\sf PA}^-_{\sf scat}}}
             \newcommand{\num}[1]{{\underline{#1}}}
             \newcommand{\nul}{{\mathfrak z}}
             \newcommand{\opv}{{\mathfrak s}}

  \DeclareMathOperator{\gtrial}{\text{\tikz[scale=.6ex/1cm,baseline=-.55ex,line width=.1ex]{
                            \draw[gray, fill = gray, fill opacity = .90] (-0.6,0) -- (1,1) -- (1,-1);}}}
      \DeclareMathOperator{\gtriar}{\text{\tikz[scale=.6ex/1cm,baseline=-.6ex,rotate = 180,line width=.1ex]{
                            \draw[gray, fill = gray, fill opacity = .90] (-0.6,0) -- (1,1) -- (1,-1);}}}
                            \newcommand{\gmut}{\mathrel{\gtriar\!\!\gtrial}}
                            
                            \DeclareMathOperator{\otrial}{\text{\tikz[scale=.6ex/1cm,baseline=-.55ex,line width=.1ex]{
                            \draw(-0.6,0) -- (1,1) -- (1,-1)--(-0.6,0);}}}
                            
 \newcommand{\gmutl}{\mathrel{\gtriar\!\!\otrial}}
  
    \newcommand{\recis}{\approx}
       \newcommand{\syneq}{\simeq}
       \newcommand{\bles}{\mathbin{<}}
\newcommand{\bleq}{\mathbin{\leq}}

\title{Essential Hereditary Undecidability}
 
\author{Albert Visser}
 \address{Philosophy, Faculty of Humanities,
                Utrecht University,
               Janskerkhof 13,
                3512BL~~Utrecht, The Netherlands}
\email{a.visser@uu.nl}
\date{\today}

\begin{document}

\keywords{interpretability, essential hereditary undecidability}

\subjclass[2010]{03F25,%relative consistency and interpretations
03F30,%first order arithmetic and fragments
03F40%g\"odel numberings and issues of incompleteness
}

\thanks{I thank Yong Cheng and Fedor Pakhomov for enlightening discussions.
I am grateful to Fedor for his kind permission to use a proof-idea of his. I thank Taishi
Kurahashi for his many suggestions for improvement of a previous version of this paper.}

\begin{abstract}
In this paper we study \emph{essential hereditary undecidability}. Theories with this property are a convenient tool to
prove undecidability of other theories.
The paper develops the basic facts concerning essentially hereditary undecidability
and provides salient examples, like a construction of \ehu\ theories due to Hanf and an example of
a rather natural \ehu\ theory strictly below {\sf R}. We discuss the (non-)interaction
of essential hereditary undecidability with recursive boolean isomorphism. 

We develop a reduction relation \emph{essential tolerance}, or, in the converse direction, \emph{lax interpretability} that 
interacts in a good way with essential hereditary undecidability. 

We introduce the class of $\Sigma^0_1$-friendly theories
and show that $\Sigma^0_1$-friendliness is sufficient but not necessary for essential hereditary undecidability.

Finally, we adapt an argument due to Pakhomov, Murwanashyaka and Visser to show that there is no interpretability
minimal \ehu\ theory. 
\end{abstract}

\maketitle

\section{Introduction}
Robinson's Arithmetic {\sf Q} has a wonderful property, to wit  \emph{essential hereditary undecidability}. 
This means that, if a theory is compatible with {\sf Q}, it is undecidable (and even
hereditarily undecidable). This property is very useful as a tool to prove that theories are undecidable.
 A classical example of this method is Tarski's proof of the undecidability of group
theory. See \cite{tars:unde53}.
As we will see, {\sf Q} shares this property with many other theories \dots\footnote{Tarski, Mostowski and Robinson, in \cite{tars:unde53}, provide the terminological
rules that dictate the systematic name \emph{essential hereditary undecidability}. They also prove that e.g. {\sf Q} has the property. However, they
do not explicitly employ the name. The systematic name was pointed out to us by Fedor Pakhomov. The first place
we found that uses the systematic name is \cite{hanf:mode65}.}

In this paper, we study  \emph{essential hereditary undecidability}. The paper is partly an exposition of some of the literature,
but it also contains original results and analyses.
We provide a number of examples of \ehu\ theories, both from the literature and new.

We connect the notion with two reduction relations: \emph{interpretability} and \emph{lax interpretability} (i.e., \emph{converse
essential tolerance}). Lax interpretability will be introduced in the present paper.
Specifically, we show that  essential hereditary undecidability is upward preserved under lax interpretability.
We study the interaction of the Tarski-Mostowski-Robinson theory {\sf R} with lax interpretability.
We show that, in a sense that generalises mutual lax interpretability, {\sf R} is equal to the false $\Sigma^0_1$-sentences.

We develop the notion of $\Sigma^0_1$-friendliness and show that it is sufficient but not necessary for
essential hereditary undecidability.

Finally, we demonstrate that there is no interpretability minimal \ehu\ theory. The proof is a minor adaptation of
the proof that there is no interpretability minimal essentially undecidable theory in \cite{viss:nomi22}.

\section{Basics} 
In this section we provide the basic facts and definitions needed in the rest of the paper.

\subsection{Theories and Interpretations}
A theory is, in this paper, an RE theory of classical predicate logic in finite signature. A theory is given by an index of an RE axiom set.
Here we confuse the sentences of a theory with numbers. We will usually work with a bijective G\"odel numbering of the
sentences. We adapt the G\"odel numbering in each case to the signature at hand.

We write $U\sls V$ for: $U$ and $V$ are theories in the same language and the set of theorems of $U$ is contained
in the set of theorems of $V$. We use $=_{\sf e}$ for: $U$ and $V$ are theories in the same language and $U$ and $V$ prove the same
theorems.

If $U$ is a theory, we write $\downp U$ for the set of its theorems and $\downr U$ for the set of its refutable sentences, i.e.,
$\downr U := \verz{\phi\mid U \vdash \neg\, \phi}$.

We take $\idt U$ to be the finitely axiomatised theory of identity for the signature of $U$.
This theory is a built-in feature of  predicate logic. However, if we work with interpretations, we need to check that it holds
for the equivalence relation posing as the identity of the interpreted theory.

An \emph{interpretation} $K$ of a theory $U$ in a theory $V$ is based  on \emph{a translation} $\tau$ of the $U$-language into the $V$-language.
Translations are most naturally thought of as translations between relational languages. A translation of a language with terms proceeds in two steps.
First we follow a standard algorithm to translate the language with terms into a purely relational language and then we apply a translation as described below.
A translation for the relational case commutes with the propositional connectives. In some broad sense, it also commutes with the
quantifiers but here there are a number of extra features. 
\begin{itemize}
\item
Translations may be more-dimensional: we allow a variable to be translated to an appropriate sequence of variables.
\item
We may have domain relativisation: we allow the range of the translated quantifiers to be some domain definable
in the $V$-language.
\item
We may even allow the new domain to be built up from pieces of possibly, different dimensions.
\end{itemize} 
A further feature is that identity need not be translated to identity but can be translated to a congruence relation.
 Finally, we may also allow parameters in an interpretation.
To handle these the translation may specify a parameter-domain $\alpha$. 

We can define the obvious identity translation of a language in itself, composition of translations and a disjunctive translation
$\tau\tupel{\phi}\nu$. E.g., in case $\tau$ and $\nu$ have the same dimension and are non-piecewise, the domain of
$\tau\tupel{\phi}\nu$ becomes $(\phi\wedge \delta_\tau(\vv x)) \vee (\neg\,\phi \wedge \delta_\nu(\vv x))$.

We refer the reader for details to e.g. \cite{viss:onq17}.

\emph{An interpretation} is a triple $\tupel{U,\tau,V}$, where $\tau$ is a translation of the $U$-language in the
$V$-language such that, for all $\phi$, if $U \vdash \phi$, then $V \vdash \phi^\tau$.\footnote{In case we have parameters
with parameter-domain $\alpha$ this becomes: $V \vdash \exists \vv x\, \alpha(\vv x)$ and,
for all $\phi$, if $U \vdash \phi$, then $V \vdash \forall \vv x\, (\alpha(\vv x) \to \phi^{\tau,\vv x})$.
See also Appendix~\ref{paraloca} for a discussion of the use parameters for interpretations of finitely axiomatised theories.}

We write:
\begin{itemize}
\item
$K:U \lhd V$ for: $K$ is an interpretation of $U$ in $V$.
\item
$U \lhd V$ for: there is a $K$ such that $K:U \lhd V$. We also write
$V\rhd U$ for: $U \lhd V$.
\item
$U \lhd_{\sf loc} V$ for: for every finitely axiomatisable sub-theory $U_0$ of $U$, we have
$U_0 \lhd V$.
\item
 $U \lhd_{\sf mod} V$ for: for every $V$-model $\mc M$, there is a translation $\tau$ from the $U$-language in the $V$-language, such that
 $\tau$ defines an internal $U$-model $\mc N = \widetilde \tau(\mc M)$ of $U$ in $\mc M$.
 \item
 We write $U \mutint V$ for: $U \lhd V$ and $V\lhd U$. Similarly, for the other reduction relations.
\end{itemize}

Given two theories 
$U$ and $V$ we form $W := U \ovee V$ in the following way.
The signature of $W$ is the disjoint union of the signatures
of $U$ and $V$ with an additional fresh zero-ary predicate $P$.
The theory $W$ is axiomatised by the axioms $P \to \phi$ if $\phi$ is
a $U$-axiom and $\neg\, P \to \psi$ if $\psi$ is a $V$-axiom. One
can show that $U\ovee V$ is the infimum of $U$ and $V$ in the 
interpretability ordering $\lhd$. This result works for all
choices of our notion of interpretation.

\subsection{Arithmetical Theories}
The  theory {\sf R},  introduced in \cite{tars:unde53}, is a primary example of an \ehu\ theory.
For various reasons, we will work with a slightly weaker version of {\sf R}. See Remark~\ref{joshsmurf} below for a
brief explanation of the difference and some background.
The language of {\sf R}, in our variant, is the arithmetical language $\mathbb A$ with $0$, {\sf S}, $+$, $\times$ and $<$. 

 Here are the axioms of {\sf R}. The underlining stands
for the usual unary numeral function.
\begin{enumerate}[{\sf R}1.]
\item
$\vdash \num m + \num n = \num{m+n}$
\item
$\vdash \num m \times \num n = \num{m\times n}$
\item
$\vdash \num m \neq \num n $, for $m\neq n$
\item
$\vdash x<  \num n \to \bigvee_{i< n} x= \num i$
\item
$\vdash  x< \num n \vee x= \num n \vee \num n < x$
\end{enumerate}
\begin{remark}\label{joshsmurf}
The original version of {\sf R} in \cite{tars:unde53} did not have $<$ in the language. Tarski, Mostowski and Robinson  used $\leq$ as a defined symbol with the following 
definition:
$u \leq t$ iff $\exists w\; w+u=t$. Their axioms are the obvious adaptation of the above ones with $\leq$ in stead of $<$.
One can employ an even weaker theory ${\sf R}_0$, where one drops Axiom {\sf R}5 and strengthens {\sf R}4 by replacing the implication by
a bi-implication. See, e.g., \cite{jone:vari83} for a discussion. Vaught, in his paper \cite{vaug:theo62}, employs an even weaker variant of
${\sf R}_0$. We note that there is a mistake in Vaught's formulation of his axioms. They need a strengthening to make everything work.
\end{remark}

An important tool in the present paper is the theory of a number. There are various ways to develop this. E.g., we can treat the numerical operations
as partial functions. Here we will employ a version using total functions. 
This version was developed by Johannes Marti, Nal Kalchbrenner, Paula Henk and Peter Fritz in 
\emph{Interpretability Project Report} of 2011, the report of a project they did  under my guidance in the Master of Logic in 
Amsterdam.\footnote{We simplified the axioms of Marti, Kalchbrenner, Henk and Fritz a bit and also implemented
three nice simplifications suggested by the referee of a previous paper.}
The language of {\sf TN} is the arithmetical language $\mathbb A$.
Here are the axioms of {\sf TN}.
\begin{enumerate}[{\sf TN}1.]
\item
$\vdash x \not < 0$
\item
$\vdash (x< y \wedge y <z ) \to x < z$
%\item
%$\vdash x \not < x$ 
\item
$\vdash x< y \vee x= y \vee y < x$
%\item
%$\vdash {\sf S}x \neq 0$
\item
$\vdash x=0 \vee \exists y\; x={\sf S}y$
%\item
%$\vdash (x <z \wedge y <z) \to ({\sf S}x = {\sf S}y \to x=y)$ 
\item
$\vdash{\sf S}x \not < x$ \label{gnipgnapsmurf}
\item
$\vdash x<y \to (x <{\sf S}x \wedge y \not < {\sf S}x)$ \label{gnagnasmurf}
\item
$\vdash x+0=x$
\item
$\vdash x+{\sf S}y ={\sf S}(x+y)$
\item
$\vdash x \times 0 = 0$
\item
$\vdash x\times {\sf S}y = x\times y + x$
\end{enumerate}

\noindent
We note that, by substituting $x$ for $y$, ${\sf TN}\ref{gnagnasmurf}$ implies $x \not < x$.
So, $<$ satisfies the axioms of a linear strict ordering with minimum 0. If follows from this fact in combination with
${\sf TN}\ref{gnagnasmurf}$ again, that $<$ is a discrete ordering and that {\sf S} does give the order successor
when applied to a non-maximal element. Moreover, by {\sf TN}\ref{gnipgnapsmurf}, if there is a maximal element, then {\sf S} maps it to itself.
 It follows that a model of  {\sf TN} is a discrete linear ordering with minimum 0. So, it  either represents a  finite ordinal
or starts with a copy of $\omega$. Moreover, on a finite domain the successor function will behave normally, except
on the maximum $\mathfrak m$, where we have ${\sf S}\mathfrak m = \mathfrak m$.

\begin{remark}
The structure $\mathbb Z$ is a model of all axioms except {\sf TN}1. Moreover, $\omega+1$, where
we cut off all operations at $\omega$ in the obvious way, is a model of all axioms except ${\sf TN}4$. 
Finally, $\mathbb Z_2$ with the ordering generated by $0<1$ is a model of all axioms except ${\sf TN}5$. 
\end{remark}

We will be interested in the theory of a witness of a $\Sigma^0_1$-sentence $\sigma$. There is a minor problem here.
Even if the witness exists as a non-maximal element of a model, the value of a term may stick out.
We can avoid this in several ways. 
We discussed one such way in our paper \cite{viss:onq17}. We follow the same strategy in the present paper.
We define \emph{pure $\Delta_0$}-formulas as follows:
\begin{itemize}
\item
$\delta ::= \bot \mid \top \mid u< v \mid 0=u \mid {\sf S}u=v \mid u+v=w \mid u\times  v=w  \mid \neg\, \delta \mid$\\ 
\hspace*{0.7cm}
$(\delta \wedge \delta) \mid (\delta \vee \delta) \mid (\delta \to \delta) \mid \forall u \bles v\, \delta \mid \exists u\bles v\, \delta
 \mid \forall u \bleq v\, \delta \mid \exists u\bleq v\, \delta$.
\end{itemize}
Here the bounded quantifiers are the usual abbreviations.

A \emph{pure} $\Sigma^0_1$-formula is of the form $\exists \vv u\, \delta$, where $\delta$ is pure $\Delta_0$.
In  \cite{viss:onq17}, we showed that every ordinary $\Sigma^0_1$-sentence can always be rewritten modulo
${\sf EA}+\mathrm B\Sigma_1$-provable equivalence to a pure one.
We call something a 1-$\Sigma^0_1$-formula if it starts with precisely one single existential quantifier.

A subtlety occurs in the treatment of substitution: consider a pure $\Sigma^0_1$-formula $\sigma$ and, e.g., a substitution of
a numeral in it, 
$\sigma[x:= \underline n\,]$. Here we will always assume that the result of substitution is
rewritten to an appropriate pure $\Sigma^0_1$ normal form.

We need the notion $z\models \phi$, where $z$ is considered as a number that models {\sf TN}, where
the arithmetical operations are cut off at $z$. We define $z\models \phi$ by $\phi^{\mf {tr}(z)}$, where
$\mf {tr}(z)$ is a parametric translation from the arithmetical language to the arithmetical language which is defined
as follows:
\begin{itemize}
\item
the domain of $\mf {tr}(z)$ is the set of $x$ such that $x \leq z$, 
\item
${\sf Z}^{\mf{tr}(z)}(w) :\iff w=0$,
\item
${\sf S}^{\mf{tr}(z)}(x,w) :\iff ({\sf S}x \leq z \wedge w={\sf S}x) \vee ({\sf S}x \not \leq z \wedge w=z)$,
\item
${\sf A}^{\mf{tr}(z)}(x,y,w) :\iff (x+y \leq z \wedge w=x+y) \vee (x+y \not \leq z \wedge w=z)$,
\item
${\sf M}^{\mf{tr}(z)}(x,y,w) :\iff (x\times y \leq z \wedge w=x\times y) \vee (x\times y \not \leq z \wedge w=z)$.
\end{itemize} 

Let $\sigma := \exists \vv x\, \delta$, where 
$\delta$ is pure $\Delta_0$ with at most the $\vv x$ free. Here $\vv x := x_0,\dots, x_{n-1}$.
We define:  
\begin{itemize}
\item
$\sigma^{\mf q} :=  \exists z\, (\exists x_0\bles z\dots \exists x_{n-1}\bles z\;  \delta \wedge z\models \bigwedge {\sf TN})$. 
\end{itemize}
We note that $\sigma^{\mf q}$ is equivalent with  $\exists z\; z\models (\bigwedge {\sf TN} \wedge \exists \vec x\,  (\delta \wedge \bigwedge_{i\bles n} {\sf S}x_i \neq x_i))$.

We will confuse  $\sigma^{\mf q}$ with the theory axiomatised by this sentence.
Since everything relevant to the evaluation of the sentence happens strictly below
$z$,  the pure $\Sigma^0_1$-sentence $\sigma$ in the context of $(\cdot)^{\mf q}$ has its usual arithmetical meaning.

The following result is easily verified.
\begin{theorem}
\begin{enumerate}[i.]
\item
$\sigma^{\mf q} \vdash \sigma$.
\item
Suppose $\sigma$ is true. Then, ${\sf R}\vdash \sigma^{\mf q}$.
\item
Suppose $\sigma$ is false. Then $\sigma^{\mf q} \vdash {\sf R}$.
\end{enumerate}
\end{theorem}

We will employ witness comparison notation. Suppose $\alpha$ is of the form $\exists x\, \alpha_0(x)$ and $\beta$ is of the form
$\exists y\, \beta(y)$.
We define:
\begin{itemize}
\item
  $\alpha<\beta := \exists x\, (\alpha_0(x) \wedge \forall y \bleq x\, \neg\, \beta_0(y))$.
  \item
$\alpha \bleq  \beta:= \exists x\, (\alpha_0(x) \wedge \forall y \bles  x\, \neg\, \beta_0(y))$.
\item
If $\gamma$ is  $\alpha< \beta$, then $\gamma^\bot$ is $\beta \leq \alpha$.
\item
 If $\delta$ is $\alpha \leq \beta$, then
$\delta^\bot$ is $\beta < \alpha$.
\end{itemize}

\noindent 
We note that witness comparisons between pure 1-$\Sigma^0_1$-formulas 
are again pure 1-$\Sigma^0_1$-formulas. The following insights are immediate.

\begin{theorem}\label{minabelesmurf}\label{knutselsmurf}
Suppose $\sigma$ and $\sigma'$ are pure $1$-$\Sigma^0_1$-sentences. Then,
\begin{enumerate}[i.]
\item
If $\sigma \leq \sigma'$, then ${\sigma'}^{\mf q} \vdash \sigma$.
\item
$(\sigma\leq \sigma')^{\mf q} \vdash \neg \, (\sigma' < \sigma)$ and $(\sigma< \sigma')^{\mf q} \vdash \neg \, (\sigma' \leq \sigma)$.
\item
Suppose $\sigma\leq \sigma'$. Then, $(\sigma'< \sigma)^{\mf q}$ is inconsistent.
Suppose $\sigma< \sigma'$. Then, $(\sigma'\leq \sigma)^{\mf q}$ is inconsistent.
\item
Suppose $\sigma$ is true. Then, if we allow piecewise interpretations, we have $\top \rhd \sigma^{\mf q}$. 
If we do not
allow piecewise interpretations, we still have $(\exists x\, \exists y \, x\neq y) \rhd \sigma^{\mf q}$.
\end{enumerate}
\end{theorem}

\begin{remark}
The reader of this paper will develop some feeling for the subtleties involved in our strategy to handle 
$\Sigma^0_1$-sentences in the context of theories of a number. See  Appendix~\ref{rosser} for an illustration of these
subtleties. 

In Taishi Kurahashi's paper \cite{kura:inco22} a somewhat different approach to obtain 
sentences with the good properties of the $\sigma^{\mf q}$ is worked out. Kurahashi's paper verifies many details in a careful way.

A different idea for the treatment of the theory of the witness of a $\Sigma^0_1$-sentence is to work with the usual definition
of $\Sigma^0_1$,
but demand that the maximum element, if there is one, is larger that a suitable function of
the G\"odel number of $\sigma$ and the maximum of the witnesses.

One can also develop theories of a number
using partial functions. This leads again to different possibilities to define the  $\sigma^{\mf q}$-like sentences. 
See, e.g., \cite{viss:vaug12} for an attempt to treat theories of a number in this style.

In a yet different approach, one develops finite versions of set theory. This idea is already discussed in \cite{vaug:theo62}.
See \cite{pakh:weak19} for a beautiful way of realising the idea.  
\end{remark}

\subsection{Recursive Boolean Isomorphism}

Two theories $U$ and $V$ are \emph{recursively boolean isomorphic} iff, there is a bijective recursive function $\Phi$, considered as a function from
the sentences of the $U$ language to the $V$-language, such that:
\begin{enumerate}[i.]
\item $\Phi$ commutes with the boolean connectives, so, e.g.,  $\Phi(\bot)= \bot$ and
$\Phi(\phi \wedge \psi)= (\Phi(\phi) \wedge \Phi(\psi))$,
\item
$U \vdash \phi$ iff $V \vdash \Phi(\phi)$.  
\end{enumerate}

\noindent
We note that it follows that, e.g.,
\begin{eqnarray*}
\Phi^{-1}(\phi'\wedge\psi') & = & \Phi^{-1}(\Phi\Phi^{-1}(\phi') \wedge \Phi\Phi^{-1}(\psi')) \\
& = & \Phi^{-1}\Phi(\Phi^{-1}(\phi') \wedge \Phi^{-1}(\psi')) \\
& = & \Phi^{-1}(\phi') \wedge \Phi^{-1}(\psi') 
\end{eqnarray*}
So $\Phi^{-1}$ is indeed the inverse isomorphism.

The demands on recursive boolean isomorphism are rather stringent. So it is good to know that the presence of an object
satisfying far weaker demands implies the presence of a recursive boolean isomorphism.

Let us write $\vdash$ of $U$-derivability and $\vdash'$ for $U'$-derivability. We also write $\sim$ for $U$-provable equivalence and
$\sim'$ for $U'$-provable equivalence. We let $\phi,\psi,\dots$ range over $U$-sentences and $\phi',\psi',\dots$ over $U'$-sentences.

Let us say that an RE relation $\mathcal E$ between numbers, considered as a relation between $U$- and $U'$-sentences, \emph{witnesses a
recursive Lindenbaum isomorphism} iff we have:
\begin{enumerate}[a.]
\item
For all $\phi$, there are $\chi$ and $\chi'$ such that $\phi \sim\chi \mathrel{\mc E}\chi'$;
\item
For all $\phi'$, there are $\chi$ and $\chi'$ such that $\chi \mathrel{\mc E}\chi'\sim' \phi'$;
\item
If $\phi_0 \mathrel{\mc E}\phi'_0$ and $\phi_1 \mathrel{\mc E}\phi'_1$, then $\phi_0\sim \phi_1$ iff $\phi'_0\sim \phi'_1$;
\item
If $\phi \sim \phi'$ and $\psi \sim\psi'$ and $(\phi\wedge \psi) \sim \chi \mathrel{\mc E} \chi'$, then
$\chi' \sim' (\phi'\wedge \psi')$. 
%By (c), we find that also $(\phi\wedge\psi) \sim\chi$.
Similarly, for the other boolean connectives. 
\end{enumerate}

The dual form of (d) follows from (a,c,d).
Suppose $\phi \mathrel{\mc E} \phi'$ and $\psi \mathrel{\mc E}\psi'$ and $\chi \mathrel{\mc E} \chi'\sim' (\phi'\wedge\psi')$. By (a) and (d),
we can find $\rho$ and $\rho'$ such that  $(\phi\wedge \psi) \sim \rho \mathrel{\mc E} \rho' \sim' (\phi'\wedge \psi')$.
It follows that $\rho' \sim \chi'$, and, hence, by (c), that $\rho \sim \chi$.

It is easy to see that if $\mc E$ witnesses recursive Lindenbaum isomorphism, then so does ${\sim} \circ  {\mc E} \circ {\sim'}$.

Let us say that a sentence is \emph{a pseudo-atom} iff it is either atomic or if it has a quantifier as main connective.

\begin{theorem}\label{lindenbaum}
Suppose $\mc E$ witnesses recursive Lindenbaum isomorphism between $U$ and $U'$. Then, we can effectively find
from an index of $\mc E$ an index of a recursive boolean isomorphism $\Phi$ between $U$ and $U'$.
\end{theorem}

\begin{proof}
This is by a straightforward back-and-forth argument. 
Suppose $\mc E$ witnesses recursive Lindenbaum isomorphism.
Without loss of generality we may assume that $\mc E={\sim} \circ  {\mc E} \circ {\sim'}$.
Let us employ enumerations of sentences
that enumerate boolean sub-sentences before sentences. 

We construct $\Phi$ in steps. 
Suppose we already have constructed \[(\phi_0,\phi'_0),\, \dots,\, (\phi_{k-1},\phi'_{k-1}).\] (Here $k$ may be 0.)
Suppose $k$ is even. Let $\phi_k$ be the first sentence in the enumeration of the $U$-sentences not among the $\phi_i$, for $i<k$.
In case $\phi_k$ is a pseudo-atom, we take $\phi'_k$ the first pseudo-atom in the enumeration of the $\psi'$ such that
$\phi_k \mathrel{\mc E} \psi'$. It is easy to see that there will always be such a pseudo-atom since we can always add vacuous quantifiers
to a sentence. If $\phi_k$ is, e.g., a conjunction, it will be of the form $(\phi_i \wedge \phi_j)$, for $i,j<k$, and we set $\phi_k' := (\phi'_i \wedge \phi'_j)$.
The case that $k$ is odd, is, of course, the dual case.

Clearly,  this construction indeed delivers a recursive boolean isomorphism. 
\end{proof}

Let us write $U\recis U'$ for $U$ is recursively isomorphic to $U'$. An important insight is that $\recis$ is a bisimulation w.r.t.
theory extension (in the same language). This means that:
\begin{description}
\item[zig]
If $U \recis V$ and $U'\sle U$, then there is a $V' \sle V$, such that $U' \recis V'$;
\item[zag]
 If $U \recis V$ and $V'\sle V$, then there is a $U' \sle  U$, such that $U' \recis V'$.
 \end{description}
 
 \begin{theorem}\label{bisim}
 $\recis$ is a bisimulation for $\sls$.
 \end{theorem}
 
 \begin{proof}
 We prove the zig case. Zag is similar. Suppose $U \recis V$ and $U' \sle U$. Let $\Phi$ be a witnessing isomorphism.
 We define $V'$ as $\verz{\Phi(\phi) \mid \phi \in U'}$.
 We have: 
 \qedright
  \begin{eqnarray*}
 V' \vdash \Phi(\psi) & \Iff & \exists U'_0 \subseteq_{\sf e,fin} U'\;\; V \vdash \bigwedge_{\phi \in U'_0} \Phi(\phi) \to \Phi(\psi) \\
 & \Iff & \exists U'_0 \subseteq_{\sf e,fin} U'\;\;  V \vdash \Phi(\bigwedge_{\phi \in U'_0} \phi \to \psi) \\
 & \Iff &  \exists U'_0 \subseteq_{\sf e,fin} U'\;\; U \vdash \bigwedge_{\phi \in U'_0} \phi \to \psi \\
 & \Iff & U' \vdash \psi
 \end{eqnarray*}
 \end{proof}

Suppose $\mathcal P$ is a property of theories.
We say that $U$ is \emph{essentially $\mathcal P$} if all consistent RE extensions
(in the same language) of $U$ are $\mathcal P$. We say that $U$ is \emph{hereditarily $\mathcal P$} if all consistent RE sub-theories
of $U$ (in the same language) are $\mathcal P$. We say that $U$ is \emph{potentially $\mathcal P$} if some consistent RE extension
(in the same language) of $U$ is $\mathcal P$. 

If $\mathcal R$ is a relation between theories the use of \emph{essential} and \emph{hereditary} and \emph{potential} always
concerns the first component aka the subject. Thus, e.g., we say that \emph{$U$ essentially tolerates $V$} meaning that 
$U$ essentially has the property of tolerating $V$. Tolerance itself is defined as potential intepretation.
So $U$ essentially tolerates $V$ if $U$ essentially potentially interprets $V$.

The following insight follows immediately from Theorem~\ref{bisim}.

\begin{theorem}\label{preservation}
Suppose $\mc P$ is a property of theories that is preserved by $\recis$.  Then, so is the complement of $\mc P$ and
the property of being essentially $\mc P$. Moreover, if $\mc Q$ is also a property of theories preserved by $\recis$, then
so is the intersection of $\mc P$ and $\mc Q$.
\end{theorem}

We will see that we do not have an extension of Theorem~\ref{preservation} to include hereditariness.

\begin{remark}
Of course, the development above of recursive boolean isomorphism is very incomplete. It should be embedded
in a presentation of  appropriate categories. However, in the present paper, we restrict ourselves to the bare necessities.
\end{remark}

\begin{remark}
Recursive boolean isomorphism is implied by sentential congruence (the interpretation equivalent of elementary equivalence).
However, it is not preserved by mutual interpretability. 
\end{remark}

Here is a truly substantial result due to
Mikhail Peretyat'kin:  \cite[Theorem 7.1.3]{pere:fini97}

\begin{theorem}[Peretyat'kin]\label{pere}
Suppose $U$ is an RE theory with index $i$.  Then, there
is a finitely axiomatised theory $A := {\sf pere}(i)$ such that there is a recursive boolean morphism $\Phi$ between 
$U$ and $A$. Moreover, $A$ and an index of $\Phi$ can be effectively found from $i$.
\end{theorem}

There is a much simpler result that is also useful. We need a bit of preparation to formulate it.
The result is due to Janiczak \cite{jani:unde53}. See also \cite{viss:nomi22}.
Let {\sf Jan} be the theory in the language with one binary relationsymbol {\sf E} with the following (sets of) axioms.\footnote{Our theory
differs slightly from the theory considered by Janiczak in that we added {\sf J}3. We did this to make the characterisation in
Theorem~\ref{janicz} as simple as possible.}
\begin{enumerate}[{\sf J}1.]
\item
{\sf E} is an equivalence relation.
\item
There is at most one equivalence class of size precisely $n$
\item
There are at least $n$ equivalence classes with at least $n$ elements.
\end{enumerate}

We define ${\sf A}_n$ to be the sentence: there exists an equivalence class of size
precisely $n+1$. It is immediate that the ${\sf A}_n$ are mutually independent over {\sf Jan}.

\begin{theorem}[Janiczak]\label{janicz}
Over {\sf Jan}, every sentence is equivalent with a boolean combination of the ${\sf A}_n$.
\end{theorem}

{\sf Jan} will not be recursively boolean isomorphic to propositional logic with countably propositional variables in our narrow sense, since,
in {\sf Jan}, there will be sentences equivalent to e.g. ${\sf A}_0$ that are not \emph{identical} to a 
boolean combination of ${\sf A}_i$. However,
 {\sf Jan}  will be recursively Lindenbaum isomorphic to propositional logic.

Let $U$ be any theory. Remember that we work with a bijective coding for the $U$-sentences.
We define ${\sf jprop}(U)$ by ${\sf Jan}$ plus all sentences of the form ${\sf A}_{\phi \wedge \psi} \iff
({\sf A}_{\phi} \wedge {\sf A}_{\psi})$, plus similar sentences for the other boolean connectives, plus all ${\sf A}_\phi$, whenever $U \vdash \phi$.
Clearly, we can effectively find an index of  ${\sf jprop}(U)$ from an index of $U$. We find:

\begin{theorem}\label{lekkerbeksmurf}
$U$ is recursively boolean isomorphic with ${\sf jprop}(U)$.
\end{theorem}

\begin{proof}
We define $\phi \mathrel{\mc E} \phi'$ iff $\phi' = {\sf A}_\phi$. It is easily seen that $\mc E$ witnesses recursive Lindenbaum isomorphism between
$U$ and ${\sf jprop}(U)$. So, $U$ and ${\sf jprop}(U)$ are recursively isomorphic, by Theorem~\ref{lindenbaum}.
\end{proof}

\subsection{Incompleteness and Undecidability}

We write ${\sf W}_i$ for the RE set with index $i$.
We define the following notions. 
We assume in all cases that $U$ is consistent and RE.
\begin{itemize}
\item 
$U$ is \emph{recursively inseparable} iff $\downp{U}$ and $\downr{U}$ are
recursively inseparable.
\item
$U$ is \emph{effectively inseparable} iff $\downp{U}$ and $\downr{U}$ are
effectively inseparable. This means that there is a partial recursive function $\Phi$ such
that, whenever $\downp{U}\subseteq {\sf W}_i$, $\downr{U}\subseteq {\sf W}_j$, and ${\sf W}_i \cap {\sf W}_j = \emptyset$,
we have $\Phi(i,j)$ converges and $\Phi(i,j) \not\in {\sf W}_i\cup {\sf W}_j$.
We can easily show that $\Phi$ can always taken to be total.
\item
$U$ is \emph{effectively essentially undecidable}, iff, there is a partial recursive $\Psi$, such that, for every consistent RE extension $V$ of $U$ with index $i$, 
we have $\Psi(i)$ converges and $\Psi(i) \not\in \downp{V} \cup \downr{V}$.
\end{itemize}

The second and third of these notions turn out to coincide. This was proven by Marian Boykan Pour-El. See \cite{pour:effe68}.

\begin{theorem}[Pour-El]\label{tuinsmurf}
A theory is effectively inseparable iff it is effectively essentially undecidable.
\end{theorem}

\noindent
Clearly,
recursively inseparable implies essentially undecidable. 
Andrzej Ehrenfeucht, in his paper \cite{ehre:sepa61}, provides an example of an essentially undecidable theory that is not
recursively inseparable. So there is no non-effective equivalent of Theorem~\ref{tuinsmurf}.

The next theorem is due to  Marian Boykan Pour-El and Saul Kripke. See \cite[Theorem 2]{pour:dedu67}.

\begin{theorem}[Pour-El {\&} Kripke]
Consider any two effectively inseparable theories $U_0$ and $U_1$. Then, $U_0$ and $U_1$ are recursively boolean isomorphic.
Moreover, an index of the isomorphism can be found effectively from the indices of the theories and the indices of the witnesses of
effective inseparability.
\end{theorem}

The following result is  \cite[Chapter I, Lemma, p15]{tars:unde53} and  \cite[Chapter I, Theorem 1]{tars:unde53}.

\begin{theorem}[Tarski, Mostowski, Robinson]\label{poco}
Suppose the theory $U$ is decidable. Then, $U$ has a complete decidable extension $U^\ast$.
In other words, decidable theories are potentially complete. As a direct consequence, potential decidability and
potential completeness coincide, or, equivalently, essential undecidability and essential incompleteness are extensionally the same.
\end{theorem}

\emph{Caveat emptor}: If we, e.g., restrict ourselves to finite extensions, the equivalence between 
essential undecidability and essential incompleteness fails. So, it is good to recognise these as different
notions even if they are extensionally the same.

The next result is fundamental is the study of hereditariness. It is  \cite[Chapter I, Theorem 5]{tars:unde53}.

\begin{theorem}[Tarski, Mostowski, Robinson]\label{finexd}
Suppose the theory $U$ is decidable and $\phi$ is a sentence in the $U$-language.
Then, $U+\phi$ is also decidable.
\end{theorem}

%\subsection{Some Notions}

%uppose $U$ is a consistent RE theory.
%We say that \emph{$U$ tolorates $V$} or $U \tol V$ iff $U$ potentially interprets $V$, in other words,
%if, for some consistent RE theory $U'\supseteq U$, we have $U'\rhd V$.
%We write $U \jump V$ for $U$ essentially tolerates $V$. 

%The notion of  \emph{tolerance} was introduced in \cite{tars:unde53} under the name of \emph{weak interpretability}\footnote{Strictly speaking:
%\emph{converse} weak interpretability.}.
%We like `tolerates' more since it is more directly suggestive of the intended meaning. Japaridze uses tolerance in a more general sense.
%See \cite{japa:tol92} and \cite{japa:weak93}, or the handbook paper  \cite{japa:logi98}.
%
%We write $\idt V$ for the finitely axiomatised theory of identity for the signature of $V$.

\section{Essential Hereditary Undecidability: A First Look}
In this section, we collect the basic facts about Essential Hereditary Undecidability and provide a
selection of examples.

\subsection{Characterisations}
We give with two pleasant characterisations of essential hereditary undecidability.

\begin{theorem}
A theory $U$ is essentially hereditarily undecidable iff,
for every $W$ in the $U$-language, if $U+W$ is consistent, then $W$ is undecidable.
\end{theorem}

\begin{proof}
This is immediate since $W$ is consistent with $U$ iff, for some consistent $V$, we have $U \sls V\sle  W$. 
\end{proof}

\noindent
We note that, more generally,  $U$ is essentially hereditarily $\mc P$ iff,
for every $W$ in the $U$-language, if $U+W$ is consistent, then $W$ is $\mc P$.\footnote{I owe this observation to
Taishi Kurahashi in email correspondence.}

We say that $V$ \emph{tolerates} $U$ if $V$ potentially interprets $U$. In other words, $V$ tolerates $U$ iff
there is a consistent $V'\sle  V$ such that $V' \rhd U$. Equivalently,  $V$ tolerates $U$ iff, there is a 
translation $\tau$ of the $U$-language into the $V$-language such $V+\downp{U}^\tau$ is consistent.
Finally, $V$ tolerates $U$ iff, there is a 
translation $\tau$ of the $U$-language into the $V$-language such $V+\idtb U\tau +U^\tau$ is consistent.\footnote{\label{voetsmurf}We
need small adaptations of the characterisations in terms of a translation in case we allow parameters.
For example, we have:
$V$ tolerates $U$ iff there is a 
translation $\tau$ such that \[V+\exists \vv x\, \alpha_\tau(\vv x) +\verz{ \forall \vv x \,( \alpha_\tau(\vv x) \to \phi^{\tau,\vv x}) \mid U \vdash \phi}\] is consistent.}

\begin{theorem}
Suppose $U$ is consistent.
The theory $U$ is essentially hereditarily undecidable iff every $V$ that tolerates $U$ is undecidable.
\end{theorem}

\begin{proof}
We treat the argument for the parameter-free case. The case with parameters only requires a few obvious adaptations.

Suppose $U$ is essentially hereditarily undecidable and  $V+\idtb U\tau +U^\tau$ is consistent. 
 Let $W$ be the theory in the $U$-language axiomatised by $\verz{\phi \mid V+\idtb U\tau \vdash \phi^\tau}$.

We find that $W \vdash \phi$ iff $V +\idtb U\tau\vdash \phi^\tau$ and that $U+W$ is consistent. Hence $W$ is not decidable. Suppose 
that $V$ is decidable. Then, $V+\idtb U\tau$ is decidable and so is $W$. \emph{Quod non}. So $V$ is undecidable.

The other direction is immediate.
\end{proof}

\subsection{Essential Hereditary Incompleteness}
Clearly, incompleteness is not the same as undecidability. However, essential incompleteness is the same as essential undecidability
(by Theorem~\ref{poco}). On the other hand,  incompleteness is always preserved to sub-theories.
So, \emph{a fortiori},
essential hereditary incompleteness coincides with essential incompleteness, which coincides with essential undecidability.
For example, the decidable theory {\sf Jan} has an essentially incomplete extension.
So, essential hereditary incompleteness and essential hereditary undecidability do \emph{not} coincide.

\subsection{Closure Properties}
We prove closure of the \ehu\ theories under interpretability infima.

\begin{theorem}\label{aquinosmurf}
\begin{enumerate}[a.]
\item
Suppose $U_0$ and $U_1$ are essentially undecidable. Then $U_0\ovee U_1$ is essentially  undecidable.
\item
Suppose $U_0$ and $U_1$ are essentially hereditarily undecidable. Then $U_0\ovee U_1$ is essentially hereditarily  undecidable.
\end{enumerate}
\end{theorem}

\begin{proof}
We just treat (b).
Let $P$ be the 0-ary predicate that `chooses' between $U_0$ and $U_1$ in $U:= U_0\ovee U_1$ and let ${\sf e}_i$ be the identical translation of the
$U_i$-language into the  $U_0\ovee U_1$-language. 
Suppose $W$ is consistent with $U$. Clearly, at least one of $U +W +P$ or $U+W+ \neg\,P$ is consistent.
Suppose  $U +W +P$ is consistent. It follows that $W$ tolerates $U_0$ as witnessed by the interpretation of $U_0$ in $U+W+P$ based on ${\sf e}_0$. So $W$ is undecidable. The other case
is similar.
\end{proof}

We show that the \ehu\ theories are upwards closed under interpretability.

\begin{theorem}\label{upsmurf}
Suppose $U$ is consistent and essentially hereditarily undecidable and $V\rhd U$. Then $V$ is essentially hereditarily undecidable.
\end{theorem}

\begin{proof}
Suppose that $U$ is \ehu\ and  $U$ is interpretable in $V$, say via $K$. Suppose further that $W$ is a theory in the $V$-language
that is decidable and consistent with $V$.
Let  $Z := \verz{ \phi \mid W+  \idtb VK \vdash \phi^K}$. It is easy to see that $Z$ is decidable and
consistent with $V$. \emph{Quod non}.

Our proof is easily adapted to the case with parameters.
\end{proof}

\noindent
Theorem~\ref{reversosmurf} of this paper will be a strengthening of this result.

\subsection{Hereditary Undecidability}
If a theory tolerates an \ehu\ theory, then it is not just undecidable, but hereditarily undecidable.

\begin{theorem}
Suppose $U$ is essentially hereditarily undecidable and that $V$ tolerates $U$. Then $V$ is hereditarily undecidable.
\end{theorem}

\begin{proof}
This is immediate from the fact that toleration is downward closed in both arguments.
\end{proof}

It would be great when the above theorem had a converse. However, the example below shows that this is not the case.
The example is a minor variation of Theorem~3.1 of \cite{hanf:mode65}. 

\begin{example}(Hanf).
We provide an example of a theory that is hereditarily undecidable but does not tolerate any essentially undecidable theory
(and, so, \emph{a fortiori} does not tolerate an \ehu\ theory).
We consider Putnam's example of a theory that is undecidable such that all its complete extensions are decidable.
See \cite[Section 6]{putn:deci53}.

We start by specifying a theory in the language of identity. Let:
\begin{itemize}
\item
$\widetilde n := \exists x_0\dots \exists x_{n-1}\,(\bigwedge_{i<j<n} x_i \neq x_j \wedge \forall y \, \bigvee_{k<n}y=x_k)$.
\end{itemize}
Let $\mc X$ be any non-recursive set. We take:
${\sf I}_{\mathcal X} := \verz{\neg\, \widetilde n \mid n\in \mathcal X}$.
Clearly, ${\sf I}_{\mathcal X}$ is non-recursive.

The theory ${\sf I}_{\mathcal X}$ has the following complete extensions: $\widetilde n$, for $n \not\in\mathcal X$ and
$\verz{\neg \,\widetilde n \mid n\in \omega}$. So there are no non-recursive complete extensions.
The theory  ${\sf I}_{\mathcal X}$ cannot be consistent with an essentially
undecidable $U$ in the same language (and, hence cannot tolerate an essentially
undecidable $V$), since ${\sf I}_{\mathcal X}+U$ would have a complete and recursive extension.
 
 We now apply Theorem~\ref{pere} (Peretyat'kin's result), to obtain a finitely axiomatised theory ${\sf J}_{\mathcal X}$ that is
  recursively boolean isomorphic to ${\sf I}_{\mathcal X}$.
 Clearly, ${\sf J}_{\mathcal X}$ will inherit the undecidability and the lack of non-recursive complete extensions from
  ${\sf I}_{\mathcal X}$. Since,
  ${\sf J}_{\mathcal X}$ is finitely axiomatised and undecidable, it will be hereditarily undecidable. 
 
 We note that the original theory ${\sf I}_{\mathcal X}$ extends the theory of pure identity in the language of pure identity. So,
 ${\sf I}_{\mathcal X}$ itself is \emph{not} hereditarily undecidable.
\end{example}

\begin{example}
We show that there are theories that are essentially undecidable and hereditarily undecidable but not essentially hereditarily undecidable.\footnote{I thank
Taishi Kurahashi for his suggestion to replace my previous concrete example by the current more general class of examples.}

Suppose $U$ is essentially hereditarily undecidable and $V$ is essentially undecidable but not hereditarily undecidable.

By Theorem~\ref{aquinosmurf}(a), we find that $U\ovee V$ is essentially undecidable. 

Suppose $W$ is a decidable sub-theory of $U\ovee V$.
Then, $W+ P$ is a sub-theory of $(U\ovee V)+P$, i.e., modulo derivability, $U+P$ in the extended language. Moreover, $W+P$ is decidable. It follows that  the consequences of $W+P$ in the
$U$-language are decidable. But these consequences are a sub-theory of $U$. A contradiction. So,  $U\ovee V$ is hereditarily undecidable.

Finally, let $Z$ be a decidable sub-theory of $V$. We extend the signature of $Z$ to the signature of $U\ovee V$ and add the axiom
 $\neg \, P$ plus axioms of the form $\forall \vv x\, R(\vv x)$, for all predicates $R$ of the $U$-signature.
The resulting theory $Z'$ is a definitional extension of $Z$ and, thus, decidable. Clearly, $Z'$ is consistent with $U \ovee V$. 
So $U\ovee V$ is not \ehu.
\end{example}

\subsection{Essentially Hereditarily Undecidable Theories}
In this subsection, we give an overview of some \ehu\ theories.

A first insight is given by Theorem~\ref{finexd} and  \cite[Chapter I, Theorem 6]{tars:unde53}.

\begin{theorem}[Tarki, Mostowski, Robinson]\label{smulsmurf}
Suppose the theory $A$ is finitely axiomatizable.
If $A$ is undecidable, then it is hereditarily undecidable. If $A$ is essentially undecidable, then $A$ is
essentially hereditarily undecidable.
\end{theorem}

Theorems~\ref{smulsmurf}, \ref{pere} and \ref{lekkerbeksmurf} give us immediately the following insight: 

\begin{theorem}\label{twoposs}
Suppose $U$ is an \textup(essentially\textup) undecidable theory. Then, there are \textup(essentially\textup) undecidable theories $U_0$ and 
$U_1$ that are recursively boolean
 isomorphic to $U$ of which the first is \textup(essentially\textup) hereditarily undecidable and the second has a decidable sub-theory.
 Indices for $U_0$ and $U_1$ can be effectively found from an index of $U$. Specifically, we can take $U_0 := {\sf pere}(i)$, where $i$ is an index of $U$
 and $U_1 := {\sf jprop}(U)$.
\end{theorem}

The use of Theorem~\ref{pere} delivers many examples of (essentially) hereditarily undecidable theories,
 Here is, for example, Theorem~3.3  of \cite{hanf:mode65}.

\begin{theorem}[Hanf]\label{funsmurf}
Let $d$ be any non-zero RE Turing degree. Then there is a finitely axiomatised \ehu\ theory $A$ of degree $d$.
\end{theorem} 

\begin{proof}
By the results of \cite{sho:degr58}, there is an essentially undecidable RE theory $U$ of degree $d$. Say it has index $i$.
Clearly, ${\sf pere}(i)$ fills the bill.
\end{proof}

Using the ideas of \cite{viss:nomi22}, we can even arrange it so that the Turing degree of every theory that interprets the theory $A$ of
Theorem~\ref{funsmurf} is $\geq d$.

The next example is  due to Cobham. This result is presented in \cite{vaug:theo62}. See also \cite{viss:onq17} for an alternative presentation.
We will  prove the result in Section~\ref{esstol}.

\begin{theorem}[Cobham]\label{cobhamsmurf}
The theory {\sf R}  is essentially hereditarily undecidable.
\end{theorem}

We have the following corollary of Theorem~\ref{funsmurf}.

\begin{corollary}
There are \ehu\ theories that do not interpret {\sf R} and, hence, there are \ehu\ theories strictly below {\sf R}.
\end{corollary}

\begin{proof}
Suppose $d$ is an RE Turing degree strictly between 0 and $0'$. By Theorem~\ref{funsmurf}, we can find an \ehu\ theory $A$ of RE degree $d$.
 If $A \rhd {\sf R}$, then the degree of $A$ would be $0'$, so, $A\nrhd {\sf R}$.
By Theorem~\ref{aquinosmurf} in combination with Theorem~\ref{cobhamsmurf}, the theory $B := A\ovee {\sf R}$ is \ehu. Moreover, since  $A\nrhd {\sf R}$, the theory $B$ is strictly below {\sf R}.

A different and more natural example is the theory ${\sf PA}^{-}_{\sf scat}$ of Section~\ref{sepasec}.
\end{proof}

A well-trodden path is the construction of essentially undecidable theories using recursively inseparable sets.
We give the basic lemma.

\begin{lemma}\label{bruikbaresmurf}
Suppose $\Phi$ is a recursive function from the natural numbers to the sentences of $U$. Let $\mathcal X,\mathcal Y$ be a pair
of recursively inseparable sets. Suppose $\Phi$ maps $\mathcal X$ to $\downp{U}$ and $\mathcal Y$ to the
$\downr{U}$. Then, $U$ is essentially undecidable. 
\end{lemma}

From the proof of Theorem~3.2 of \cite{hanf:mode65} we can extract the following analogue of  Lemma~\ref{bruikbaresmurf}
for the case of \ehu\ theories.

\begin{lemma}[Hanf]\label{handigesmurf}
Let $U$ be a consistent RE theory and let $U_0$ be a finitely axiomatised sub-theory of the $U$.
Suppose $\Phi$ is a recursive function from the natural numbers to the sentences of $U$. Let $\mathcal X,\mathcal Y$ be a pair
of recursively inseparable sets. 
Suppose $\Phi$ maps $\mathcal X$ to  $U_{0\mf p}$ and $\mathcal Y$ to 
$\downr{U}$. Then, $U$ is \ehu. 
\end{lemma}

\begin{proof}
Let $U,U_0,\Phi,\mathcal X,\mathcal Y$ be as in the statement of the theorem.
Suppose $W$ is a theory in the language of $U$ that is consistent with $U$. Suppose 
$W$ is decidable. By Theorem~\ref{finexd}, we find that $W^\ast := W+U_0$ is decidable.
Moreover, $W^\ast$ is consistent with $U$.
We have:
\begin{eqnarray*}
n \in \mathcal X  & \To & U_0 \vdash \Phi(n) \\
& \To & W^\ast \vdash \Phi(n).\\
%&&\\  
m\in \mathcal Y & \To & U \vdash \neg\, \Phi(m) \\
& \To & W^\ast \nvdash \Phi(m).
\end{eqnarray*}
It follows that $\verz{k \mid W^\ast \vdash \Phi(k)}$ is decidable and  separates $\mathcal X$ and $\mathcal Y$.
A contradiction.
\end{proof}

As we will see, in Section~\ref{esstol}, the essential hereditary undecidability of the salient theory {\sf R} is directly connected with
the essential hereditary undecidability of certain finitely axiomatised theories. The following example, due to Hanf
in \cite[Theorem 3.2]{hanf:mode65}, shows that there are very un-{\sf R}-like \ehu\ theories.
\begin{example}(Hanf).\label{nofinex}
We produce an \ehu\ RE theory $U$ that does not tolerate any finitely axiomatisable essentially undecidable theory $A$.

Let $\mathcal X$ and $\mathcal Y$ be recursively inseparable sets. Let $V := {\sf Jan} + \verz{{\sf A}_n \mid n\in \mathcal X}$.
Let $B$ be ${\sf pere}(i)$, where $i$ is an index of $V$, and let $\Psi$ be the boolean isomorphism from $V$ to $B$.
We define  ${\sf B}_i:= \Psi({\sf A}_i)$ and
$U := B + \verz{\neg\, {\sf B}_j \mid j \in \mathcal Y}$.
 By Lemma~\ref{handigesmurf}, the theory $U$ is \ehu. 
 
 Suppose  $U$ tolerates a finitely axiomatised essentially undecidable theory $A$. Then, some finite theory $C$ in the language of $U$ is consistent
 with $U$ and interprets $A$. Clearly, $C$ must itself be essentially undecidable. Now $\Psi^{-1}(C)$ is equivalent to a boolean combination of the
 ${\sf A}_i$ over $V$, so $C$ is equivalent to a boolean combination of the ${\sf B}_i$ over $B$. Let the set of the $i$ so that ${\sf B}_i$ occurs in this boolean combination be $\mathcal F$.
 Let $W := B +C + \verz{{\sf B}_i \mid i \not\in \mathcal F}$. Clearly, $W$ is consistent and decidable. A contradiction with the fact that $C$ is essentially undecidable.
 
 We note that we can get our example in any desired non-zero RE Turing degree by choosing the appropriate $\mathcal X$ and $\mathcal Y$.
\end{example}

In \cite{viss:frie22}, we show that effectively Friedman-reflexive theories are \ehu. 
We state it here as a theorem. The theorem will be a direct consequence of Theorem~\ref{luiesmurf} of this paper.

\begin{theorem} \label{timmersmurf}
 Suppose $U$ is consistent, RE, and effectively Friedman-reflexive. Then, $U$ is  essentially hereditarily undecidable.
 \end{theorem}

\section{Essential Tolerance and Lax Interpretability}\label{esstol}
In this section we study a reduction relation that interacts very well with essential hereditary undecidability.
We will prove a number of theorems that illustrate these connections.

\subsection{Basic Definitions and Facts}
Suppose $U$ is a consistent RE theory.
We remind the reader  that \emph{$U$ tolerates $V$}, or $U \tol V$, iff $U$ potentially interprets $V$, in other words,
if for some consistent RE theory $U'\sle  U$, we have $U'\rhd V$.
We find that $U$ \emph{essentially tolerates} $V$ iff $U$ essentially potentially interprets $V$, explicitly: iff, for all consistent RE theories $U'\sle  U$, there is a consistent RE theory
 $U'' \sle  U'$, such that
$U'' \rhd V$. We write $U \jump V$ for $U$ essentially tolerates $V$. 

We note that essential tolerance is analogous to the converse of interpretability. In other words, `essentially tolerates' is analogous to
`interprets'. We will call the converse of essential tolerance: \emph{lax interpretability}.

Below we establish that essential tolerance is a \emph{bona fide} reduction relation---unlike tolerance that fails
to be transitive.

\begin{remark}
The notion of  \emph{tolerance} was introduced in \cite{tars:unde53} under the name of \emph{weak interpretability}.
We like `tolerates' more since it is more directly suggestive of the intended meaning. Japaridze uses tolerance in a more general sense.
See \cite{japa:tol92} and \cite{japa:weak93}, or the handbook paper  \cite{japa:logi98}.
\end{remark}

\begin{example}
We illustrate the intransitivity of tolerance. In fact, our counterexample shows a bit more.

Presburger Arithmetic essentially tolerates Predicate Logic in the language with a binary relation symbol.
Predicate Logic in the language of a binary relation symbol tolerates full Peano Arithmetic. However,
Presburger Arithmetic does not tolerate Peano Arithmetic.
\end{example}

\begin{remark}
The definition of $\jump$ suggests several variations, where we demand that some promised ingredients are effectively found from appropriate
indices. We 
will not explore such variations in the present paper.
\end{remark}

\begin{remark}
Robert Vaught, in his paper \cite{vaug:theo62} introduces a notion that we would like to call \emph{parametrically local interpretability} or $\lhd_{\sf pl}$.
This notion interacts in desirable ways with essential hereditary undecidability. We discuss the relationship between $\jumpb$ and $\lhd_{\sf pl}$ in
Appendix~\ref{paraloca}. We show that $\jumpb_{\sf c}$,  a slightly improved version of $\jumpb$, satisfies: if $U \rhd_{\sf pl}V$, then $U\jump_{\sf c}V$.
Moreover,  for our purposes, $\jumpb_{\sf c}$ retains all the good properties of $\jumpb$.
\end{remark}

The first two insights are that lax interpretability is (strictly) between two good notions of interpretability, to wit, model interpretability and
local interpretability.

\begin{theorem}
If $U \rhd_{\sf mod} V$, then $U \jump V$.
\end{theorem}

\begin{proof}
Suppose $U \rhd_{\sf mod} V$. Let $U'$ be a consistent theory with $U' \sle  U$. Consider any model $\mathcal M$ of $U'$. There is an
$\mc M$-internal model of $V$, say, given by translation $\tau$. Let $U'' := U'+ \verz{\svb^\tau \mid V \vdash \svb}$.
Clearly, $U''$ is consistent and RE and $U'' \rhd V$ as witnessed by $\tau$.
\end{proof}

In Section~\ref{sepasec}, we develop the  theory
$\pams$. This theory is a sub-theory of {\sf R}.  We have  $\pams \jump {\sf R}$, but $\pams \nrhd_{\sf mod} {\sf R}$.
This tells us that the inclusion of model interpretability in lax interpretability is strict.

\begin{question}
Are there sequential $U$ and $V$ such that we have $U \jump V$, but $U \nrhd_{\sf mod}V$?
\end{question}

We turn to the comparison of lax and local interpretability.

\begin{theorem}
If $U \jump V$, then $U \rhd_{\sf loc} V$.
\end{theorem}

\begin{proof}
Suppose $U \jump V$. Let $V_0$ be a finitely axiomatised sub-theory of $V$. Let $\sva$ be a single axiom of $V_0$ which
includes $\idt V$. Suppose $U \nrhd V_0$. Consider $U' := U + \verz{\neg\, \sva^\tau \mid \tau: \Sigma_V \to \Sigma_U}$.\footnote{In case we
are allowing parameters, we should replace $\phi^\tau$ in the argument by 
\[\exists \vv v\,\alpha_\tau(\vv v)\wedge \forall \vv u\, (\alpha_\tau(\vv u) \to \phi^{\tau,\vv u}).\]
Here $\alpha_\tau$ is the parameter domain. In the case of finitely axiomatised theories, we can even
omit the parameter domain, so  $\exists \vv u\; \phi^{\tau,\vv u}$ already works.}
The theory $U'$ is consistent since, if not, $U$ would prove a finite disjunction of sentences
of the form $\sva^\tau$. Say the translations involved are $\tau_0,\dots,\tau_{n-1}$. We define:
\[ \tau^\ast := \tau_0\tupel{\sva^{\tau_0}}(\tau_1\tupel{\sva^{\tau_1}}(\dots (\tau_{n-2}\tupel{\sva^{\tau_{n-1}}}\tau_{n-1})\dots)).\] 
We find that $U \vdash \sva^{\tau^\ast}$. \emph{Quod non}. So, $U'$ is consistent. Clearly $U'$ is RE
and no consistent RE extension of $U'$ can interpret the theory axiomatised by $\sva$. But this contradicts $U \jump V$.
\end{proof}

\begin{example}
Consider a consistent finitely axiomatised sequential theory $A$. 
We do have $A \rhd_{\sf loc} \mho(A)$. Here $\mho(A)$ is the theory ${\sf S}^1_2+\verz{{\sf Con}_n(A) \mid n\in \omega}$,
where ${\sf Con}_n$ means consistency w.r.t. proofs where all formulas in the proof have depth of quantifier alternations
complexity $\leq n$. See, e.g., \cite{viss:seco11} for more on $\mho$.

In, e.g., \cite{viss:fefe14} it is verified in detail that $A$ has a consistent RE extension $\widetilde A$
such that every interpretation of ${\sf S}^1_2$ in $A$ contains a restricted inconsistency statement for $A$.
We call such an extension \emph{a Kraj\'{\i}\v{c}ek-theory based on $A$}.
Clearly, no consistent extension of $\widetilde A$ in the same language can interpret $\mho(A)$.
So $A \not\jump \mho(A)$. This gives us our desired separating example between $\jump$ and $\rhd_{\sf loc}$.

We note that $A \not\jump \mho (A)$ in fact expresses the existence of a Kraj\'{\i}\v{c}ek extension.

Another example is as follows. Consider any complete and decidable theory $U$. We do have $U \rhd_{\sf loc} {\sf R}$. However, $U \not \jump {\sf R}$. Since
no complete RE theory does interpret {\sf R}.
\end{example}

It turns out that it is useful to lift $\jump$ to a relation between sets of theories. 
We define:
\begin{itemize}
\item
$\mathcal X \jump \mathcal Y$ iff for all $U\in \mathcal X$ and for all consistent  RE theories $U'\sle  U$,
there is a consistent RE theory $U''\sle  U$ and a $V\in \mathcal Y$, such that $U'' \rhd V$.
\end{itemize}

We note that $U \jump V$ is equivalent to $\verz U \jump \verz V$. We will write $U \jump \mathcal Y$ for
$\verz U \jump \mathcal Y$, etcetera.

\begin{theorem}
\begin{enumerate}[a.]
\item
Suppose $\mathcal X \subseteq \mathcal X'$ and $\mathcal X'\jump\mathcal Y'$ and  $\mathcal Y' \subseteq \mathcal Y$.
Then, $\mathcal X \jump \mathcal Y$.
\item
We have: $\mathcal X\jump\mathcal Y$ and $\mathcal X'\jump\mathcal Y$ iff $(\mathcal X \cup \mathcal X')\jump\mathcal Y$.
\item
The relation $\jump$ between sets of theories is transitive. As a consequence, $\jump$ as a relation between theories is transitive.
\end{enumerate}
\end{theorem}

\begin{proof}
We just treat (c).
Suppose $\mathcal X \jump \mathcal Y\jump \mathcal Z$.
Consider $U \in \mathcal X$ and let $U'$ be any consistent RE extension of $U$. Let $U''$ be a consistent RE extension
of $U'$  such that $U'' \rhd V$, for some $V \in \mathcal Y$. Say, we have $K:U'' \rhd V$. Let $V' := \verz{\svb\mid U'' \vdash \svb^K}$.
We find that $V'$ is a consistent RE extension of $V$. Let $V''$ be a consistent extension of
$V'$  such that $V'' \rhd W$, for some $W\in \mathcal Z$.

We consider $U^\ast := U'' + \verz{\svb^K \mid V'' \vdash \svb}$. Clearly $U^\ast$ is RE, $U^\ast \sle  U'$ and $U^\ast \rhd V'' \rhd W$, so $U^\ast \rhd W$.
We claim that $U^\ast$ is consistent. If not, there would be a $\svb$ such that 
$V'' \vdash \svb$ and $U'' \vdash (\neg\, \svb)^K$. It follows, by the definition of $V'$,
 that $V' \vdash \neg\, \svb$ and, hence, that $V'' \vdash \neg \, \svb$, contradicting the fact that
$V''$ is consistent. Thus, $U^\ast\sle  U'$ and $W$ are our desired witnesses.
\end{proof}

We write $U \tol \mathcal Y$ for $U$ tolerates some element of $\mathcal Y$.
Inspection of the above proof also tells us that:

\begin{theorem}
Suppose $U \tol V$ and $V \jump \mathcal Z$. Then, $U \tol \mathcal Z$.
\end{theorem}

\begin{theorem}
\begin{enumerate}[i.]
\item
$(U \ovee V) \jump W$ iff $U \jump W$ and $V \jump W$.
\item
$(U\ovee V)\gmut \verz{U,V}$.
\end{enumerate}
\end{theorem}

\begin{proof}
We just do (i). Claim (ii) is similar.
From left-to-right is immediate, since $U \rhd (U \ovee V)$ and $V \rhd (U \ovee V)$, and, hence,
$U \jump (U \ovee V)$ and $V \jump (U \ovee V)$. So, we are done by transitivity.

Let  $Z:= U \ovee V$. Suppose  $U\jump W$ and $V \jump W$. Let $Z'\sle  Z$ be RE and consistent. The theory
$Z'$ is either consistent with $P$ or with $\neg P$. Suppose it is consistent with $P$.
Let $U'$ be the set of $U$-sentences that follow from $Z'+P$. Clearly, $U'\sle  U$ and $U'$ is RE and consistent.
So, there is a $U''\sle  U'$ that is RE and consistent such that $U'' \rhd W$. We take $Z''$ the theory axiomatised by $Z'+P+U''$ in the $Z$-language.
Clearly, $Z''\sle  Z'$ and $Z''$ is consistent and RE and $Z'' \rhd W$. The argument in case $\neg\, P$ is consistent is similar. 
\end{proof}

We note that the above theorem tells us that the embedding functor of interpretability into lax interpretability preserves infima.

\begin{remark}
We define $\boxvee$ as follows. $U \boxvee V$ is the result of taking the disjoint union of the sigmatures of $U$ and $V$ and taking as axioms
$\phi\vee \psi$, whenever $U \vdash \phi$ and $V \vdash \psi$. It is easy to see that $\boxvee$ gives representatives of the infimum for
$\lhd_{\sf mod}$, $\jumpb$, and $\lhd_{\sf loc}$. 
\end{remark}

\begin{question}
It would be good to have a counterexample that shows that $\boxvee$ is not generally an infimum for $\lhd$. 
\end{question}

\begin{question}
The new notion of \emph{lax interpretability} raises many questions. E.g.:
is there a good supremum for lax interpretability? And: does the  embedding functor of interpretability into lax interpretability have a
right or left adjoint?
\end{question}

\subsection{Essential Hereditary Undecidability meets Lax Interpretability} 

We start with the main insight concerning the relation between  Essential Hereditary Undecidability and Lax Interpretability.

\begin{theorem}\label{reversosmurf}
Let $U$ be RE and consistent.
\begin{enumerate}[i.]
\item
Suppose $\mathcal V$ is a class of essentially undecidable theories and $U\jump \mathcal V$.
 Then, $U$ is essentially undecidable.
\item
Suppose $\mathcal V$ is a class of \ehu\ theories and $U\jump \mathcal V$.
 Then, $U$ is essentially hereditarily undecidable.
 \end{enumerate}
\end{theorem}

\begin{proof}
Ad (i).
Suppose $\mathcal V$ is a class of essentially undecidable theories and $U\jump \mathcal V$.
Suppose $U$ has a consistent decidable extension $W$. then, $W$ has a decidable consistent
complete extension $W^\ast$. It follows that $W^\ast \rhd V$, for some $V$ in $\mc V$. \emph{Quod impossibile}.

\medent
Ad (ii).
Suppose $\mathcal V$ is a class of \ehu\ theories and $U\jump \mathcal V$.
Suppose $W$ is an RE theory in the language of $U$ and suppose $U' := U \cup W$ is consistent.
We have to show that $W$ is undecidable.

Let $U''$ be a consistent RE extension of $U'$ such that $K:U'' \rhd V$, for some $V\in \mathcal V$.
Consider the theory $Z:=  \verz{\svb \mid W+\idtb VK \vdash \svb^K}$.
We have $Z \vdash \svb$ iff  $W+\idtb VK \vdash \svb^K$. 
Clearly $Z$ is a sub-theory of  $V$. 
If $W$ were decidable then so would $Z$, contradicting the fact that $V$ is 
essentially hereditarily undecidable.\footnote{We need small adaptations of the argument
in case we allow parameters. See also Footnote~\ref{voetsmurf}.}
\end{proof}

\begin{remark}
Inspection of the example provided by Ehrenfeucht in \cite{ehre:sepa61}, shows that his construction 
provides an example where $U \jump \mc V$, each element of $\mc V$ is recursively inseparable (if we wish, even effectively inseparable),
but $U$ is not recursively inseparable. The theory $U$ of the example is essentially undecidable.
\end{remark}

We now turn to the result that motivates looking at classes of theories a relata of  $\rhd$.
Let $\spth S$ be the set of all theories $\sigma^{\mf q}$, where $\sigv$ is a false pure $\Sigma^0_1$-sentence and $\sigma^{\mf q}$ is consistent.\footnote{We can also
allow the inconsistent theory in $\spth S$.}
\begin{theorem}\label{sigmasmurf}
We have ${\sf R} \gmut \spth S$.
\end{theorem}

\begin{proof}
From left-to-right. Let $U'$ be a consistent RE extension of {\sf R}. Clearly, $U' \vdash \sigma^{\mf q}$, for all true pure $\Sigma^0_1$-sentences
$\sigv$. So, if no $\sigma^{\mf q}\in \spth S$, would be consistent with $U'$, we could decide $\Sigma^0_1$-truth. \emph{Quod non}.
Consider any such $\sigma^{\mf q}$ that is consistent with $U'$. Let $U'' := U'+\sigma^{\mf q}$. Clearly, $U'' \rhd \sigma^{\mf q}$.

From right-to-left. Consider $\sigma^{\mf q}\in \spth S$. Clearly, $\sigma^{\mf q} \rhd {\sf R}$ and we are easily done.
\end{proof}

\noindent
Of course, the extension $U''$ in the proof of Theorem~\ref{sigmasmurf} can be found effectively from an index of $U'$.
We outline one way to do it.

\begin{proof}[Sketch of an alternative proof of Theorem~\ref{sigmasmurf}]
Let $\opr$ be $U'$-provability. By the G\"odel Fixed Point Lemma, we can find a pure $\Sigma_1$-sentence $\jmath$ that is equivalent to $\opr\neg\,\jmath^{\sf q}$.\footnote{We
need a bit of careful attention to ensure that our sentence is pure.}
Suppose $\jmath$ were true, then we have both $\opr\jmath^{\sf q}$ and $\opr\neg\,\jmath^{\sf q}$, contradicting the consistency of 
$U'$. So $\jmath$ is false and $\jmath^{\sf q}$ is consistent with $U'$. We take $U'' := U'+\jmath^{\sf q}$.
\end{proof}

Since all $\sigma^{\mf q}\in \spth S$ extend {\sf R}, they are essentially undecidable. Moreover, since the $\sigma^{\mf q}$ are finitely axiomatised, they
are \ehu. It follows from Theorem~\ref{reversosmurf}, that {\sf R} is \ehu. So, this gives us a proof of Theorem~\ref{cobhamsmurf}. 

\begin{question}\label{nozelsmurf}
Suppose $U \jump \spth S$. Does it follow that $U$ is recursively inseparable?
\end{question}

 Let $\spth F$ be the set of all finitely axiomatised \ehu\ theories. Example~\ref{nofinex} shows that
 there is an \ehu\ theory $U$ such that $U \not\jump \spth F$.

\section{$\Sigma^0_1$-friendliness and $\Sigma^0_1$-representativity}
In this section, we have a brief look at a rather natural property of theories that implies essential hereditary undecidability.

Consider a consistent RE theory $U$ and a
 recursive function $\Phi$ from pure 1-$\Sigma^0_1$-sentences to $U$-sentences. We give three possible properties of
 $\Phi$. Let $\sigv$ range over pure 1-$\Sigma^0_1$-sentences.
 \begin{enumerate}[$\Sigma$1.]
 \item
  If $\sigv$ is true, then $U \vdash \Phi(\sigv)$.
  \item
   $(U+\Phi(\sigv))\rhd \sigma^{\mf q}$.
\item
Suppose $\sigma\leq \sigma'$. Then, $U \vdash \neg\, \Phi(\sigma'<\sigma)$. 
 Similarly, suppose $\sigma< \sigma'$. Then, $U \vdash \neg\, \Phi(\sigma'\leq \sigma)$.
 \end{enumerate}
 We say that $U$ is \emph{$\Sigma^0_1$-friendly} iff, there is a recursive $\Phi$ satisfying $\Sigma 1$ and
 $\Sigma 2$. We say that $U$ is  \emph{$\Sigma^0_1$-representative} if  there is a recursive $\Phi$ satisfying $\Sigma 1$ and
 $\Sigma 3$.
 
 The next theorem is, in a sense, a generalisation of the First Incompleteness Theorem. We just need $\Sigma 1$.
 
 \begin{theorem}\label{ge-een}
Consider a consistent RE theory $U$ and a
 recursive function $\Phi$ from pure  1-$\Sigma^0_1$-sentences to $U$-sentences. 
 Suppose $\Phi$ satisfies $\Sigma 1$. Let $U_i$ be a recursive sequence of
consistent RE extensions of $U$. Then, we can effectively find a false pure 1-$\Sigma^0_1$-sentence $\jmath$, such that 
$\Phi(\jmath)$ is consistent with each of the $U_i$.
 \end{theorem}
 
 \begin{proof}
 We stipulate the conditions of the theorem.
We can find a pure $\Delta_0$-formula $\pi(i,p,\phi)$ such that  $U_i\vdash \phi$ iff $\exists p\; \pi(i,p,\phi)$.
 Let \[\apr\phi := \exists u\, \exists i\bles u\, \exists p\bles u\;\, \pi(i,p,\phi).\]
 
 Using the G\"odel Fixed Point Construction, we find a pure 1-$\Sigma^0_1$-sentence $\jmath$ such that $\jmath$ is true 
 iff $\apr\neg\, \Phi(\jmath)$.
  Suppose $\jmath$ is true.
 Then, $U \vdash \Phi(\jmath)$ and, for some $i$, we have $U_i \vdash \neg\, \Phi(\jmath)$,
 contradicting the consistency of $U_i$. Thus, $\jmath$ is false and consistent with each of the $U_i$.
 \end{proof}
 
 \begin{remark}(Kripke).
 We immediately get Kripke's version of the First Incompleteness Theorem from Theorem~\ref{ge-een}.
 Let {\sf School} be the theory in the language of arithmetic (without $<$ as primitive) of all true closed equations. We get Kripke's result by
 setting $U := {\sf School}$ and $\Phi$ the transformation promised by Matiyasevich's theorem that sends a pure 1-$\Sigma^0_1$-sentence
 to a purely existential sentence.
 
 E.g., it follows that there is a Diophantine equation that has solutions in all finite rings and in some non-standard model of {\sf PA}, but no
 solutions in $\mathbb N$.
 \end{remark}
 
 \begin{lemma}
 Every $\Sigma^0_1$-friendly theory $U$ is $\Sigma^0_1$-representative.
 \end{lemma} 
 
 \begin{proof}
 Let $\Phi$ witness that $U$ is $\Sigma^0_1$-friendly. We prove $\Sigma 3$. Let $\sigma$ and $\sigma'$ be pure 1-$\Sigma^0_1$-sentences.
We have:
 $(U+\Phi(\sigma'<\sigma)) \rhd [\sigma'< \sigma]$.  Suppose $\sigma\leq \sigma'$. Since, by Theorem~\ref{minabelesmurf},
 $[\sigma'< \sigma]$ is inconsistent, we find that $U\vdash \neg\,\Phi(\sigma'<\sigma)$.
The other case is similar. 
 \end{proof}
 
\begin{theorem}\label{vettesmurf}
Suppose $U$ is RE, consistent, and $\Sigma^0_1$-friendly. Then, $U\jump \spth S$, and, hence,  $U\jump {\sf R}$.
\end{theorem}

\begin{proof}
We note that any consistent RE extension of a $\Sigma^0_1$-friendly RE theory is again $\Sigma^0_1$-friendly. So it is sufficient to show that
$U$ tolerates a theory $\sigma^{\mf q}$, for false pure 1-$\Sigma^0_1$-sentences $\sigv$.

Suppose $U$ does not tolerate any false  $\sigma^{\mf q}$. If $\sigv$ is true, we have $U \vdash \Phi(\sigv)$.
Suppose $\sigv$ is false. We have $(U+\Phi(\sigv)) \rhd \sigma^{\mf q}$. So, if $U+\Phi(\sigv)$ were consistent,
then $U$ would tolerate $\sigma^{\mf q}$. \emph{Quod non, ex hypothesi}. So, $U \vdash \neg\,\Phi(\sigv)$.
Since $U$ is RE, we can now decide the halting problem. \emph{Quod impossibile}.
\end{proof}

We note that we can effectively find a sentence $\sigma$ such that $\Phi(\sigma)$ is consistent with $U$ from indices for
$U$ and $\Phi$. Let $\opr\neg\,\Phi(s)$ be a pure 1-$\Sigma^0_1$-formula representing the $U$-provability of $\neg \,\Phi(s)$.
Then, we can take $\sigma$ to be $\jmath$, (a pure 1-$\Sigma^0_1$ version of) the G\"odel fixed point  that is equivalent to
$\opr\neg\,\Phi(\jmath)$. It is easy to see that $U+\Phi(\jmath)$ is indeed consistent.
 
The following result employs the notions and notations of \cite{viss:frie22}.
 
 \begin{theorem}\label{luiesmurf}
 Every consistent RE effectively Friedman-reflexive theory $U$ is $\Sigma^0_1$-friendly.
 \end{theorem}
 
 \begin{proof}
 We can take $\Phi(\sigma) := \ccz\sigma$
 \end{proof}
 
 We note that  Theorems~\ref{vettesmurf} and \ref{luiesmurf} immediately give Theorem~\ref{timmersmurf}.

The theory {\sf R} is $\Sigma^0_1$-friendly via the mapping $\sigv \mapsto \bigwedge \sigma^{\mf q}$. So this again shows that ${\sf R} \jump \spth S$.

It turns out that $\Sigma^0_1$-representativity coincides with a familiar notion.

 \begin{theorem}\label{reptoei}
Consider a consistent RE theory $U$. Then, $U$ is $\Sigma^0_1$-representative iff $U$ is effectively inseparable. 
\end{theorem}

\begin{proof}
Suppose $U$ is RE and consistent.

\medent
Suppose $U$ is $\Sigma^0_1$-representative as witnessed by $\Psi$. Let $\mc X_0$ and $\mc X_1$ be any pair of effectively inseparable sets.
Let $\sigma_0(x)$ be a pure 1-$\Sigma^0_1$-formula that represents $\mc X_0$ and 
let $\sigma_1(x)$ be a pure 1-$\Sigma^0_1$-formula that represents $\mc X_1$.
We write $\sigma_i(\underline n)$ for a pure 1-representation of the result of substituting $\underline n$ in $\sigma_i$.
We define \[\Theta(n) := \Psi(\sigma_0(\underline n) \leq \sigma_1(\underline n)).\]
Suppose $n\in \mc X_0$. Then,  $\sigma_0(\underline n) \leq \sigma_1(\underline n)$ is true
and, hence, $U \vdash \Theta(n)$. Suppose $n \in \mc X_1$. Then,
$\sigma_1(\underline n) < \sigma_0(\underline n)$ is true, and, hence, 
$U \vdash \neg\,\Theta(n)$.
 
 \medent
 For the converse, suppose $U$ is effectively inseparable. Then, by \cite[Theorem 2]{pour:dedu67}, we find that there is a recursive boolean isomorphism $\Psi$ from
  {\sf R} to $U$. We can take $\Psi$ restricted to pure 1-$\Sigma^0_1$-sentences
  as the function witnessing the $\Sigma^0_1$-representativity of $U$.
 \end{proof}
 
The first part of the proof of Theorem~\ref{reptoei} can also be done via a Rosser argument. We have to be somewhat more careful
with the details if we follow that road. We will give the argument in Appendix~\ref{rosser}.
 
 \begin{question}
It would be quite interesting to replace the $\sigma^{\mf q}$ in our definitions of friendliness and representativity by some other class of theories. However, the demands on
the $\sigma^{\mf q}$ use
both witness comparison and truth. So, it is not at all obvious here what more general analogues could be. 
\end{question}

\begin{example}
At this point the time is ripe to give some separating examples. We consider properties: {\sf P}1: undecidable, {\sf P}2: essentially undecidable, {\sf P}3:  essentially hereditarily undecidable,
{\sf P}4: recursively inseparable, {\sf P}5: effectively inseparable, {\sf P}6: $\Sigma^0_1$-friendly. We first give the list and then the description of
the examples below it.

\[
\begin{tabular}{|c||c|c|c|c|c|c|} \hline
example & {\sf P}1 & {\sf P}2  & {\sf P}3 & {\sf P}4 & {\sf P}5 & {\sf P}6 \\ \hline \hline
$U_0$ & $-$ & $-$ & $-$ & $-$ & $-$ & $-$ \\ \hline
$U_1$ & $+$ & $-$ & $-$ & $-$ & $-$ & $-$ \\ \hline
$U_2$ & $+$ & $+$ & $-$ & $-$ & $-$ & $-$ \\ \hline
$U_3$ & $+$ & $+$ & $+$ & $-$ & $-$ & $-$ \\ \hline
$U_4$ & $+$ & $+$ & $-$ & $+$ & $-$ & $-$ \\ \hline
$U_5$ & $+$ & $+$ & $-$ & $+$ & $+$ & $-$ \\ \hline
$U_6$ & $+$ & $+$ & $+$ & $+$ & $+$ & $-$ \\ \hline
$U_7$ & $+$ & $+$ & $+$ & $+$ & $+$ & $+$ \\ \hline
\end{tabular}
\]

\medskip
\begin{enumerate}[a.]
\item
We can take $U_0$ be any decidable theory like Presburger Arithmetic.
\item
We can take $U_1$ e.g. the theory of groups, which, by results of Tarski (\cite[Chapter III]{tars:unde53}) 
and Szmielev (\cite{szmi:elem55}), is hereditarily but not essentially undecidable.
\item
We can take $U_2$ to be Ehrenfeucht's theory (see \cite{ehre:sepa61}) which is essentially undecidable, but neither hereditarily undecidable, nor recursively inseparable.
\item
We can take $U_3$ to be a finitely axiomatised theory that is recursively boolean isomorphic to $U_2$. This theory is essentially hereditarily undecidable, but
not recursively inseparable.
\item
Let $d$ be an RE Turing degree with $0< d<0'$.
Suppose  $\mathcal A$, $\mathcal B$ is a recursively inseparable pair of RE sets as constructed by Shoenfield (see \cite{sho:degr58} or \cite{viss:nomi22}), where the 
Turing degree of  $\mathcal A$ is $d$ and the Turing degree of  $\mathcal B$ is $\leq d$. Let
 \[U_4:= {\sf Jan}+\verz{{\sf A}_n \mid n\in \mathcal A}+ \verz{\neg\, {\sf A}_n \mid n\in \mathcal B}.\]
 Then, $U_4$ is recursively inseparable, but cannot be effectively inseparable. Also, since $U_4$ contains a decidable sub-theory in the
 same language it cannot  be \ehu.
 \item
 We define $U_5$ like $U_4$ only now we take $\mathcal A$ and $\mathcal B$ to be effectively inseparable.
  \item
  We can take $U_6$ to be the theory of Hanf's example (Example~\ref{nofinex} in this paper) for the case that the recursively inseparable sets on which the construction is based
  are effectively inseparable.
  \item
  We can take $U_7$ to be, e.g., {\sf Q}.
  \end{enumerate}
  
  \noindent
  We note that our list shows that the evident dependencies of the concepts are all possible dependencies.
\end{example}

\begin{question}
Is there a  finitely axiomatised and effectively inseparable theory that is not $\Sigma^0_1$-friendly?
\end{question}

\section{Separating Model-Interpretability and Lax Interpretability}\label{sepasec}
In this section, we introduce the theory $\pams$ and prove some of its salient properties.
Most importantly, it will be an example of a $\Sigma^0_1$-friendly theory $U$ such that $U$ is a sub-theory of {\sf R} and $U \jump {\sf R}$, and $U \nrhd_{\sf mod} {\sf R}$.

\begin{remark}
Our theory $\pams$ is closely related to Vaught's theory {\sf S}. See \cite{vaug:theo62}.
\end{remark}

We define the theory ${\sf PA}^-_{\sf scat}$ as follows. It has the relational signature of arithmetic with $<$ minus the zero.
We write $\widetilde n(a)$ for $(\exists z\, z=\underline n)^{\mf q}$ with 0 replaced by the parameter $a$ plus `there are at least $n$ elements' relativised to 
the domain of the $y$ such that $a\leq y$.
We note that number and theory fix each other uniquely.
${\sf PA}^-_{\sf scat}$ is axiomatised by the axioms $\exists a\;\widetilde n(a)$ for $n\in \omega$.

Let ${\sf R}_{\sf succ}$ be the theory in the language with 0 and {\sf S}, axiomatised by $\underline n \neq\underline m$, where $n\neq m$.
We prove the following theorem. 
\begin{theorem}\label{grootmachtigesmurf}
We have:
\begin{enumerate}[a.]
\item
${\sf R} \sle  \pams$.
\item
$\pams$ is $\Sigma^0_1$-friendly and, hence, ${\sf PA}^-_{\sf scat}  \jump {\sf R}$.
\item
${\sf PA}^-_{\sf scat}\nrhd_{\sf mod}{\sf R}_{\sf succ}$.
\end{enumerate}
\end{theorem} 
\noindent We note that it follows that $\pams \gmutl {\sf R}$, and, hence, $\pams\gmut {\sf R}$.

We prove our theorem via a sequence of lemmas.
The following lemma is clear.

\begin{lemma}
${\sf R} \sle  \pams$.
\end{lemma}

For any pure 1-$\Sigma^0_1$-sentence, $\sigv$ we 
define $\sigma^{\mf q}_x$ as the theory of $\sigv$ with  zero replaced by the free parameter $x$ on the domain of the $y$ such that $x\leq y$. We show the following lemma.

\begin{lemma}
The theory ${\sf PA}^-_{\sf scat}$ is $\Sigma^0_1$-friendly. Hence, ${\sf PA}^-_{\sf scat} \jump \spth S$ and ${\sf PA}^-_{\sf scat}\jump {\sf R}$.
\end{lemma}

\begin{proof}
It is easy to see that the mapping $\sigv \mapsto \exists x\, \sigma^{\mf q}_x$ fulfils the conditions for $\Phi$ in the definition of $\Sigma^0_1$-friendlyness.\footnote{This uses
interpretations with parameters. At the end of the section, we explain how to obtain the result, for a variant of
${\sf PA}^-_{\sf scat}$, in a parameter-free way. Here we loose the fact that our theory is contained in ${\sf R}_{<}$, but we still have that {\sf R} interprets it.}
\end{proof}

To prove Theorem~\ref{grootmachtigesmurf}(c), we need a counter-model.
We define $\mathbb N_{\sf scat}$, the model of the scattered numbers, for the signature of ${\sf PA}^-_{\sf scat}$ as follows.
It is the disjoint sum of the natural numbers considered as models of the theory of a number (in relational signature).
Modulo isomorphism, we can also define  $\mathbb N_{\sf scat}$ concretely as follows.
{\small
\begin{itemize}
\item
The domain of $\mathbb N_{\sf scat}$ is the set of $\tupel{n,m}$ with $m<n$.
\item
$\tupel{n,m} < \tupel{n',m'}$ iff $n=n'$ and $m<m'$.
\item
${\sf S}(\tupel{n,m},\tupel{n',m'})$ iff $n=n'$ and $m' = {\sf min}(m+1,n-1)$.
\item
${\sf A}(\tupel{n,m},\tupel{n',m'},\tupel{n'',m''})$ iff $n=n'=n''$ and $m'' = {\sf min}(m+m',n-1)$.
\item
${\sf M}(\tupel{n,m},\tupel{n',m'},\tupel{n'',m''})$ iff $n=n'=n''$ and $m'' = {\sf min}(m\times m',n-1)$.
\item
$\tupel{n,m} < \tupel{n',m'}$ iff $n=n'$ and $m<m'$.
\end{itemize}
}

Let $\sim$ be the equivalence relation on $\mathbb N_{\sf scat}$ given by $x<y \vee y \leq x$. 
We write $\braee x$ for  (the purely syntactic representation of) the $\sim$-equivalence class of $x$. Let  $\phi^{\braee x}$ be the relativisation of the 
quantifiers of $\phi$ to $\braee x$.\footnote{To be precise: we first rename the variables in $\phi$ so that there
is no bound variable labeled $x$ and then relativise.} 
We write $\imod n$ for the submodel (of the theory of $n$) with as domain the $\sim$-equivalence class of size $n$. 

Consider a finite set of free variables $X$. Let $\syneq$ be any equivalence relation on  $X$. So, $\syneq$ is a purely syntactic
relation.
We say that a formula is $\theta$ is  \emph{$\syneq$-good} if it is of the form $\chi^{\braee x}$, for some $x\in X$ and all free variables of
$\theta$ are $\syneq$-equivalent to $x$.  We say that a formula is \emph{$\syneq$-friendly} iff it is a boolean combination of $\syneq$-good formulas
of the form $\chi^{\braee x}$, where $x\in X$.
We define ${\sf E}_{\syneq}$ to be the conjunction of all $x\sim x'$ in case $x\syneq x'$
 and  $x\not\sim x'$ in case $x\not \syneq x'$, for $x,x'\in X$. (We assume that $X$ is part of the data for $\syneq$.)

We prove our third lemma. The lemma and its proof are well known for the case of binary disjoint sums of models.
The present result is just an adaptation.

\begin{lemma}
Let $X$ be some finite set of variables and let $ \vv x$ be some enumeration of $X$.
Consider a formula $\phi$ with free variables among $X$ and an equivalence relation $\syneq$ on $X$.
Then, there is an $\syneq$-friendly $\psi$ such that
\[\mathbb N_{\sf scat}\models \forall  \vv x\, ({\sf E}_{\syneq} \to (\phi \iff \psi)).\]
\end{lemma}

\begin{proof}
The proof is by induction on $\phi$.

We treat ${\sf A}xyz$ as a prototypical atomic formula, where {\sf A} stands for addition. 
In case ${\sf A}xyz$ is $\syneq$-good, we can take ${\sf A}xyz$ itself as $\psi$ noting that
  ${(\sf A}xyz)^{\braee x}$ is identical to ${\sf A}xyz$. In case ${\sf A}xyz$ is not $\syneq$-good, we can take $\psi$ to be
  $\bot$, seeing that ${\sf A}xyz$ is equivalent to $\bot$ under the assumption ${\sf E}_{\syneq}$ and $\bot = \bot^{\braee x}$.

We treat  $(\phi_0 \wedge \phi_1)$ as a prototypical case.
In ${\sf Th}(\mathbb N_{\sf scat})$, under the assumption ${\sf E}_{\syneq}$, by the Induction Hypothesis,
each of the $\phi_i$ is equivalent to an $\syneq$-friendly $\psi_i$. So, 
$(\phi_0 \wedge \phi_1)$ is equivalent to 
  $(\psi_0 \wedge \psi_1)$, which is $\syneq$-friendly.

  Finally, consider the case of $\exists z\, \phi'$. 
  Since we always can rename bound variables, we can assume that $z\not \in X$.
  Let $X' := X \cup\verz z$. We write ${\syneq} \sqsubset_z {\syneq'}$ if $\syneq'$ is an
  equivalence relation on $X'$ and $\syneq'$ restricted to $X$ is $\syneq$. 
  
  We reason in ${\sf Th}(\mathbb N_{\sf scat})$ under the assumption ${\sf E}_{\syneq}$.
  We note that $\exists z\, \phi'$ is equivalent to $ \bigvee_{{\syneq} \sqsubset_z {\syneq'}} \exists z\,({\sf E}_{\syneq'} \wedge \phi')$.
  We zoom in on some $\alpha :=({\sf E}_{\syneq'} \wedge \phi')$. By the induction hypothesis, this can be rewritten
  as the conjunction of ${\sf E}_{\syneq'}$ and a disjunction of formulas
  of the form $(\theta_0 \wedge \theta_1)$, where $\theta_0$ is $\syneq$-friendly and
  $\theta_1$ is $\syneq'$-good and of the form $\chi^{\braee y}$, where $y \syneq z$. If the $\syneq'$-equivalence class of
  $z$ contains at least two elements, we choose $y$ different from $z$. 
Given our assumption that ${\sf E}_{\syneq}$, the formula 
  $\alpha$ is equivalent to a disjunction of formulas of the form $(\theta_0 \wedge ({\sf E}_{\syneq'} \wedge \theta_1))$.
  It follows that $\exists z\ \alpha$ is equivalent to
   a disjunction of formulas of the form $(\theta_0 \wedge \exists z\, ({\sf E}_{\syneq'} \wedge \theta_1))$. 
   It clearly suffices to show that $ \exists z\, ({\sf E}_{\syneq'} \wedge \theta_1)$ is equivalent to
   an $\syneq$-friendly formula.
   
  There are two cases. The $\syneq'$-equivalence class of $z$ contains at least two elements or precisely one.
  
   In the first case we can replace $\exists z\, ({\sf E}_{\syneq'} \wedge \theta_1)$ by
   $\exists z\in \braee y\, \theta_1$, where $\theta_1$ is of the form $\chi^{\braee y}$. Clearly,  $\exists z\in \braee y\, \theta_1$ is $\syneq$-good.
   
    In the second case, at most $z$ occurs freely in $ \theta_1$. 
    Suppose $X=\verz{x_0,\dots,x_{k-1}}$. Let $\theta_1$ be $\chi^{\braee z}$.
    In the context of ${\sf E}_{\syneq}$, we can rewrite $\exists z\,({\sf E}_{\syneq'} \wedge \theta_1)$ to the formula
    $\exists z\, (\bigwedge_{i<k} z\not\sim x_i \wedge \theta_1)$.  
    In case $\chi^{\braee z}$ can be 
    fulfilled in more than $k$ submodels corresponding to numbers, $\exists z\, (\bigwedge_{i<k} z\not\sim x_i \wedge \theta_1)$
    will be true. So we can replace  $\exists z\, (\bigwedge_{i<k} z\not\sim x_i \wedge \theta_1)$
     by $\top$. If not $\chi(z)$ will be fulfilled in precisely  $\imod n_0, \dots, \imod n_{s-1}$.
    Let 
    \[  \card_n(u) := \exists v_0 \dots \exists v_{n-1}\, (\bigwedge_{i<j<n} v_i \neq v_j \wedge
     \forall w\, (w = u \iff \bigvee_{i<n} w=v_i)).\]
    We find that  $\exists z\, (\bigwedge_{i<k} z\not\sim x_i \wedge \theta_1)$ is equivalent to
    $\bigvee_{j<s} \bigwedge_{i<k}  \neg\, \card^{\braee {x_i}}_{n_j}(x_i)$, which is clearly $\syneq$-friendly.\footnote{Note that
    if $s=0$, this formula becomes $\bot$.}
\end{proof}

\begin{lemma}
There is no inner model of ${\sf R}_{\sf succ}$ in $\mathbb N_{\sf scat}$. 
\end{lemma}

\begin{proof}
Since $\mathbb N_{\sf scat}$ has at least two elements, we do not need to consider piece-wise interpretations. Moreover,
every element in $\mathbb N_{\sf scat}$ is definable. So, we can always eliminate parameters. Thus it is sufficient to prove
our result for many-dimensional relativised interpretations without parameters.

Suppose we had an inner model of ${\sf R}_{\sf succ}$ given by an interpretion $M$. Say $M$ is $m$-dimensional 
and suppose $0$ is given by a formula $\nul( \vv x\,)$ and {\sf S} by
$\opv ( \vv x, \vv y\,)$.  We note that in the sequence $ \vv x, \vv y$ all the variables are pairwise disjoint. There are two conventional 
aspects. The variables in $\opv$ need only be among the $ \vv x, \vv y$, but not all need to occur. Secondly,
the order of the variables $ \vv x$, $ \vv y$ as exhibited need not be given by anything in $\opv$.

Our proof strategy is to obtain a contradiction by finding 
a finite set of numbers $\mathcal N$ and an infinite sequence of pairwise different sequences length $m$  with components
in the $\imod n$ for $n\in \mathcal N$.
We work in $\mathbb N_{\sf scat}$. 
\begin{enumerate}[i.]
\item
We fix an $m$-sequence $ \vv a$ such that $\nul( \vv a\,)$, in other words $ \vv a$ represents $0^M$.
We put the $n_i$ such that $\card_{n_i}(a_i)$ in $\mathcal N$.
\item
For each equivalence relation on the elements of $ \vv x, \vv y$ 
we add a set of numbers to $\mathcal N$. Consider a relation $\syneq$ on the elements of $ \vv x, \vv y$. 
Under the assumption ${\sf E}_\syneq$, we can rewrite
$\opv$ as $\bigvee_{q< r}\bigwedge_{p<s_q}\theta_{qp}$, where $\theta_{qp}$ is $\syneq$-good, say, it is  of the form $(\chi_{qp})^{\braee{w_{qp}}}$.
We can clearly arrange it so that (i) $w_{qp}$ is always the first in the sequence $ \vv x, \vv y$ of its $\syneq$-equivalence class and (ii)
if $p<p'$, then $w_{qp}$ occurs strictly earlier in $ \vv x, \vv y$ than $w_{qp'}$. So, if $p\neq p'$, we have $w_{qp} \not\syneq w_{qp'}$.
Consider any $\theta_{qp}$ where $w_{qp}$ is one of the $y_i$. There are two possibilities. 
\begin{enumerate}[I.]
\item
Suppose the number of $\imod n$ in which $\chi_{qp}$ is satisfiable is
$< 2m+1$. In this case we add all $n$ such that $\chi_{qp}$ is satisfiable in $\imod n$ to $\mathcal N$. 
\item
 Suppose the number of $\imod n$ in which $\chi_{qp}$ is satisfiable is
$\geq 2m+1$. In this case we add the first $2m+1$ such $n$ to $\mathcal N$.
\end{enumerate}
\item
Nothing more will be in $\mathcal N$.
\end{enumerate}

Let $\mathcal N^\ast$ be the elements in the $\imod n$, for $n\in \mathcal N$.
Clearly $\mathcal N$ is finite and so is the number of elements in $ \mathcal N^\ast$. 

We are now ready and set to define our infinite sequence in order to obtain the desired contradiction.
The sequence starts with $ \vv a$. We note that $ \vv a$ is in the domain of $M$ and that the components of $ \vv a$ are
in $\mathcal N^\ast$. Each element of the sequence will be in the domain of $M$ and its components will be in $\mathcal N^\ast$.
Suppose we have constructed the sequence up to $ \vv b$. Since $ \vv b$ is in the domain of $M$, there is a $ \vv c$ with
$\opv( \vv b, \vv c\,)$. We define $\syneq$ on the elements of $ \vv x, \vv y$ as follows.
We will say that $b_i$ is the value of $x_i$ and $c_j$ is the value of $y_j$.
Let $ \vv d =  \vv b, \vv c$ and $ \vv v =  \vv x, \vv y$. We take $v_i \syneq v_j$ iff $d_i \sim d_j$.
We note that we have
${\sf E}_\syneq[ \vv v: \vv d\,]$. We construct the formula  $\bigvee_{q< r}\bigwedge_{p<s_q}\theta_{qp}$ as before for $\syneq$.
So, for some $q<r$, we have $\bigwedge_{p<s_q}\theta_{qp}[ \vv v: \vv d\,]$.

Consider the variable $w_{qp}$. If it is an $x_i$, then all variables $y_i$ that are $\syneq$-equivalent to it, will have values that are
$\sim$-equivalent to $b_i$. So they will be in $\mathcal N^\ast$. If it is a $y_i$ and we are in case
(ii.I) of the construction of $\mathcal N$ the values of the variables equivalent to it
will be in $\mathcal N^\ast$. 
The final case is that  $w_{qp}$ is a $y_i$ and we are in case (ii.II) of the construction of $\mathcal N$.
Since there are at least $2m+1$ numbers $n$ such that $\imod n$ satisfies $\chi_{qp}$, we can always choose
an $n^\ast$ among these numbers such that no $d_i$ is in $\imod n^\ast$ such that $\imod n^\ast$ satisfies $\theta_{qp}$.
We assign to $y_i$ in the equivalence class of $w_{qp}$ the value $e_i$ so that under this assignment $\chi_{q,p}$ is satisfied.
We now modify our sequence $ \vv b, \vv c$ by replacing the $c_j$ by the $e_j$ for the cases where $y_j$ is in 
the equivalence class of $w_{qp}$. Say the new sequence is $ \vv b, \vv{c'}$. We note that the new sequence has strictly less
elements outside $\mathcal N^\ast$ and that we still have $\opv( \vv b, \vv {c'}\,)$.
We repeat this procedure for all $w_{qp}$ that are among the $y_i$. The final sequence we obtain will only have values in $\mathcal N^\ast$.
 
 By the axioms of ${\sf R}_{\sf succ}$, we we cannot have two elements in our sequence that are the same. A contradiction.
\end{proof}

We end this section by describing how we can make the result work for parameter-free interpretations.
A first step is to modify the definition of $\sigma^{\mf q}$, say, to $\sigma^{\mf q\ast}$. We remind the reader that we assume our
$\sigv$ are in pure form. In the definition of $\sigma^{\mf q}$ we just asked for there to be a witness of $\sigv$. 
For $\sigma^{\mf q\ast}$ we ask that the witness $w$ is the smallest one and that $w+1$ is the maximum element.

We now define ${\sf PA}^-_{\sf scat}!$ as the theory axiomatised by $\exists !x\,\sigma^{\mf q\ast}(x)$.
We note that $\mathbb N_{\sf scat}$ also satisfies ${\sf PA}^-_{\sf scat}!$.
The new theory is not a sub-theory of ${\sf R}_{<}$. However, the theory is locally finite, i.e., every finitely axiomatised sub-theory has a finite model.
So, by the main result of \cite{viss:whyR14}, we have ${\sf R} \rhd {\sf PA}^-_{\sf scat}!$.\footnote{We can also provide an interpretation is a
more direct way by mimicking the definition of $\mathbb N_{\sf scat}$ in {\sf R}.}
All our other arguments work with $\exists !x\,\sigma^{\mf q\ast}(x)$ replacing $\exists x\, \sigma^{\mf q}(x)$.
We note that  ${\sf PA}^-_{\sf scat}! + \exists !x\,\sigma^{\mf q\ast}(x)$ interprets $\sigma^{\mf q\ast}$ in a parameter-free way.
E.g., the definition of the domain becomes: $\delta(y) := \exists x\, (\sigma^{\mf q\ast}(x) \wedge x\leq y)$.

  \section{Non-Minimality}
  There is no interpretability minimal essentially hereditarily undecidable theory. There is a quick proof of this due to Fedor Pakhomov and there is
  a slow proof. Since, the slow proof yields different information, I do reproduce it here. See also Remark~\ref{smurferdiesmurf}.
  
  Here is the quick proof. 
  The proof is a minor adaptation of the proof of \cite[Theorem 1.1]{viss:nomi22} as given in Section~4.4 of that paper.
  
  \begin{theorem}\label{megasmurf}
  There is no interpretability minimal essentially hereditarily undecidable RE theory.
  \end{theorem}
  
  \begin{proof}
  Since,  by Theorem~\ref{aquinosmurf}, the \ehu\ theories are closed under interpretability suprema, it is sufficient to show that
  there is no \emph{minimum} essentially hereditarily undecidable RE theory. Suppose, to obtain a contradiction, that
   $U^\star$ is such a minimum theory.
  
  Let $i$ be an index of an RE set. By a result of Shoenfield \cite{sho:degr58}, we can effectively find an index $j$ of the theory ${\sf sh}(i)$ such that
  ${\sf W}_i$ is not recursive iff ${\sf sh}(i)$ is essentially undecidable. 
  See also  \cite[Theorem 4.8]{viss:nomi22}. Next, by a result of Peretyat'kin \cite{pere:fini97},
   we can effectively find an index $k$ of a theory ${\sf pere}(j)$ that is finitely axiomatized and recursively Boolean isomorphic with
   ${\sf sh}(i)$. 
    Let us call this theory
  ${\sf shpe}(i)$. Since ${\sf shpe}(i)$ is essentially undecidable and finitely axiomatized it is essentially hereditarily undecidable.
  Let {\sf Rec} be the set of indices of recursive sets. We have:
  \begin{eqnarray*}
  i \not\in {\sf Rec} & \text{iff} & {\sf shpe}(i) \text{ is hereditarily essentially undecidable} \\
  & \text{iff} &  {\sf shpe}(i) \rhd U^\star.
  \end{eqnarray*}
  
  By a result of Rogers and, independently, Mostowski, {\sf Rec} is complete 
  $\Sigma^0_3$. See \cite[Chapter 14, Theorem XVI]{roge:theo67} or \cite[Corollary 4.3.6]{soar:turi16}.
 We have reduced the complement of {\sf Rec}, a $\Pi^0_3$-complete predicate, to an
  interpretability statement: a $\Sigma^0_3$-predicate. \emph{Quod impossibile}.
  \end{proof}
  
  We prove the non-minimality result w.r.t. interpretability for essentially hereditarily undecidable theories again using the idea behind the construction from the
  proof of \cite[Theorem~3.2 ]{hanf:mode65}, following the plan of the proof of \cite[Theorem 1.1.]{viss:nomi22} as given in Section~4.2 of that paper.
  
  We will need a variation on Kleene's well known construction of two effectively inseparable sets.
  We write $x\cdot y$ for Kleene application. For $i=0,1$, 
  let ${\sf Km}_i := \verz{\tupel{n,x} \mid x\cdot\tupel{n,x} \simeq i}$.

\begin{lemma}\label{trivismurf}
Suppose $\mathcal W$ is a recursive set. Let $\Theta$ be a 0,1-valued recursive function such that $\Theta(x)=1$ iff $x\in \mc W$.
We can  find an index $c$ of $\Theta$ effectively from an index $i$ of $\mc W$.
 Then, for any $n$, we have  $\tupel{n,c}\in   {\sf Km}_0\setminus \mc W$ or $\tupel{n,c}\in {\sf Km}_1 \cap \mc W$.
\end{lemma}  
 
\begin{proof}
We have:
\qedright
\begin{eqnarray*}
\tupel{n,c} \not\in \mc W & \Iff  & c\cdot \tupel{n,c} \simeq 0 \\
& \Iff & \tupel{n,c}\in {\sf Km}_0\\
\tupel{n,c} \in \mc W & \Iff  & c\cdot \tupel{n,c} \simeq 1 \\
& \Iff & \tupel{n,c}\in {\sf Km}_1
\end{eqnarray*}
\end{proof}

\begin{lemma}\label{bromsnorsmurf}
Suppose $\mc Z$ is an RE set such that, for every $m$, there is an $n$ such that $\tupel{n,m}\in \mc Z$.
Then ${\sf Km}_0\cap \mc Z$ and  ${\sf Km}_1\cap \mc Z$ are effectively inseparable. 
\end{lemma}

\begin{proof}
Suppose $\mc W$ with index $i$ is decidable and that $\mc W$ separates ${\sf Km}_0\cap \mc Z$ and  ${\sf Km}_1\cap \mc Z$.
Let $\Theta$ and $c$ be as in Lemma~\ref{trivismurf}. We find $n$ such that $\tupel{n,c}\in \mc Z$. By Lemma~\ref{trivismurf}, we have
 $\tupel{n,c}\in   {\sf Km}_0\setminus \mc W$ or $\tupel{n,c}\in {\sf Km}_1 \cap \mc W$. In the first case $\tupel{n,c}\in   ({\sf Km}_0\cap \mc Z)\setminus \mc W$.
 \emph{Quod non}, by our assumption that $\mc W$ separates  ${\sf Km}_0\cap \mc Z$ and  ${\sf Km}_1\cap \mc Z$. In the second case, we have
 $\tupel{n,c}\in {\sf Km}_1\cap \mc Z \cap \mc W$, again contradicting the assumption. 
\end{proof}

\begin{theorem}\label{supersmurf}
Consider any essentially undecidable RE theory $U$. Then, we can effectively find \textup(an index of\textup) an 
essentially hereditarily undecidable RE theory
$V$ such that $V\nrhd U$. Moreover, we can take $V$ to be effectively inseparable.
\end{theorem}

\begin{proof}
Let $T := {\sf Jan}+\verz{{\sf A}_n \mid n\in {\sf Km}_0}$. Let $s$ be an index of $T$. We take $A := {\sf pere}(s)$. So, $A$ is finitely axiomatised and
  recursively boolean isomorphic to $T$. Let $\Phi$ be the witnessing recursive isomorphism from $V$ to $A$ and let ${\sf B}_i := \Phi({\sf A}_i)$.
Clearly, over $A$, every sentence is provably equivalent to a boolean combination of the ${\sf B}_i$.

Let ${\sf C}_{n,0},\dots, {\sf C}_{n,2^n-1}$ be an enumeration of all conjunctions
of $\pm {\sf B}_i$, for $i<n$.
Suppose $U$ is an essentially undecidable RE theory.
Let $\upsilon_0,\upsilon_1,\dots$ be an effective enumeration of the theorems
of $U$. Let $\tau_0,\tau_1,\dots$ be an effective enumeration of all translations from
the $U$-language into the $A$-language.\footnote{We will do the argument for the case
of parameter-free translations. By a minor adaptation, we can add parameters.}

Consider $n$, $\tau_i$ and ${\sf C}_{n,j}$, for $j<2^n$. Let ${\sf V}_{n,j} :=
A+{\sf C}_{n,j}+\verz{{\sf B}_k \mid k\geq n}$. Clearly, ${\sf V}_{n,j}$ is
either inconsistent or consistent and complete.
We claim that, for some $k$, we have ${\sf V}_{n,j} \vdash \neg\, \upsilon^{\tau_i}_k$.
Suppose this were not the case. Then, ${\sf V}_{n,j}$ is consistent and
 $\tau_i$ carries an interpretation of $U$ in ${\sf V}_{n,j}$, but this is impossible since
 ${\sf V}_{n,j}$ is decidable and $U$ is essentially undecidable.

Thus, we can effectively find a number $p_{n,i,j}$ as follows.
We find the first $k$ such that ${\sf V}_{n,j} \vdash \neg\, \upsilon^{\tau_i}_k$. 
Then, we reduce,  the sentence $\upsilon^{\tau_i}_k$ to a boolean combination of
${\sf B}_s$ over $A$. 
Let $p_{n,i,j}$ the smallest number of the form $\tupel{r,n}$ that is strictly larger
 than the $s$ such that ${\sf B}_s$ occurs in this boolean combination.

We define $\eta(n,i)$ to be the maximum of the $p_{n,i,j}$, for $j<2^n$.
Let $\Psi(0) := 0$ and let $\Psi(k+1):= \eta(\Psi(k),k)$. Clearly, $\Psi$ is recursive and strictly increasing.
Let $\mathcal Z$ be the range of $\Psi$. The set $\mathcal Z$ is obviously recursive.

We define $V: = A+\verz{\neg\,{\sf B}_i \mid i \in {\sf Km}_1 \cap \mathcal Z}$.
Suppose, to obtain a contradiction, that we have $K:V \rhd U$.
Let the underlying translation of $K$ be $\tau_{n^\ast}$.

Clearly, there is a $j^\ast$ such that 
$V+{\sf C}_{\Psi(n^\ast),j^\ast}$ is consistent. (This is a non-constructive step.)
By construction, there is a
 $\phi$ with $U \vdash \phi$ and ${\sf V}_{\Psi(n^\ast),j^\ast}\vdash\neg\, \phi^K$. Moreover, 
   $A \vdash \phi^K \iff \rho$, where $\rho$ is a
 boolean combination   of ${\sf B}_s$ with $s<\Psi(n^\ast)$.
We note that no  $\neg\, {\sf B}_r$ with $\Psi(n^\ast) < r < \Psi(n^\ast+1)$ occurs in the axiomatisation of $V$. 
By our assumption on $K$, we have $V \vdash \phi^K$ and, so, $V \vdash \rho$. It follows that 
\[ A+\verz{\neg\,{\sf B}_i \mid i \in {\sf Km}_1 \cap \mathcal Z \text{ and } i<\Psi(n^\ast+1)} \vdash \rho.\]
Hence, also ${\sf V}_{\Psi(n^\ast),j^\ast} \vdash \rho$, i.e., ${\sf V}_{\Psi(n^\ast),j^\ast} \vdash \phi^K$,
A contradiction. 

We verify that $V$ is  essentially hereditarily undecidable and effectively inseparable. 
We note that $\Xi$ with $\Xi(n) := {\sf B}_n$ maps ${\sf Km}_0$ into $\downp{A}$ and
${\sf Km}_1\cap \mc Z$ into $\downr{V}$. It is immediate from Lemma~\ref{bromsnorsmurf}
that ${\sf Km}_0$ and ${\sf Km}_1\cap \mc Z$ are effectively inseparable, hence so is
$V$. Finally, by Lemma~\ref{handigesmurf}, we find that $V$ is essentially hereditarily undecidable.
\end{proof}

Since the essentially hereditarily undecidable RE theories are closed under interpretability infima,
we again obtain Theorem~\ref{megasmurf} from Theorem~\ref{supersmurf}.

\begin{remark}\label{smurferdiesmurf}
Yong Cheng notes that the argument of \cite[Section 4.2]{viss:nomi22} yields more information 
concerning the possible classes for which the no minimality result holds. See \cite{chen:mini22}.
His insights can be applied to  our case. 
For example, we find that Theorem~\ref{supersmurf}  tells us that there
is no interpretability minimal element among essentially hereditarily undecidable theories that is
also effectively inseparable. 
\end{remark}

%\bibliographystyle{alpha}
%\bibliography{provint}

\appendix

\section{Parametrically Local Interpretability}\label{paraloca}
In this appendix, we present a reduction notion, to wit \emph{c-lax interpretability}, that  improves upon both
lax interpretability and parametrically local interpretability w.r.t. the treatment of essential hereditary undecidability.

\subsection{Parameters and Finite Axiomatisability}
Our preferred treatment of interpretations with parameters uses a parameter domain $\pi$.
Let $V$ be the interpreting theory and let $U$ be the interpreted theory and let $\tau$ be the translation.
We demand that (a) $V \vdash \exists \vv p\, \pi(\vv p)$ and (b) if 
$U \vdash \phi$, then $V \vdash \forall \vv q\, (\pi (\vv q) \to \phi^{\tau,\vv q})$.

In case $U$ is finitely axiomatised, we can take a short-cut and drop the parameter domain. Let $\alpha$ be the
conjunction of the axioms of $U$ plus the axioms of the theory of identity.
We ask: (a$'$) $V \vdash \exists \vv p\; \alpha^{\tau, \vv p}$.

We note that (a) and (b) immediately imply (a$'$). Conversely, if we have (a$'$), we can define a parameter
domain $\pi^\ast(\vv p) :\iff \alpha^{\tau,\vv p}$ and use the equivalence between axioms- and theorems-interpretability
to obtain (a) and (b) for $\pi^\ast$.

We note that, if we start with a parameter-domain $\pi$ that satisfies (a) and (b), then $\pi^\ast$ is an extension of $\pi$.

\subsection{Parametric Local Interpretability}
We say that  \emph{$V$ parametrically locally interprets $U$}, or $V \rhd_{\sf pl} U$, iff, there is a parametric translation $\tau$ of the $U$-language to the
$V$-language, such that, whenever $U \vdash \phi$, then $V \vdash \exists \vv p\; \phi^{\tau,\vv p}$.
It is easy to see that $\rhd_{\sf pl}$ is transitive, using the known properties of parametric interpretability.

\begin{theorem}
Let $U$ be a theory in a language $\mc L$ extended with constants $\vv c$. 
Let $U^\ast$ be the theory axiomatised by the $\mc L$-consequences of $U$.
Then, $U \mutint_{\sf pl} U^\ast$.
\end{theorem}

\begin{proof}
We interpret $U$ in $U^\ast$ via $\tau$ which is the identical translation on the vocabulary of $\mc L$ and
which interprets constant $c_i$ as parameter $p_i$.
\end{proof}

\subsection{c-Essential Undecidability}
Consider a theory $U$. We write $U^{(\vv c)}$ for the theory $U$ with its signature expanded with constants $\vv c$.
Tarski in \cite{tars:unde53} calls this \emph{an inessential extension}.

We write $U \subseteq_{\sf c} V$ for: for some $\vv c$, we have $U^{(\vv c)} \sls V$.
We use \emph{c-essentially} in the obvious way.

\begin{theorem}
\begin{enumerate}[a.]
\item
$U$ is c-essentially undecidable iff $U$ is essentially undecidable. 
\item
$U$ is c-essentially hereditarily undecidable iff $U$ is essentially hereditarily undecidable. 
\end{enumerate}
\end{theorem}

\begin{proof}
\emph{We prove \textup(a\textup).}
The left-to-right direction is trivial. We treat right-to-left. Suppose $U$ is essentially undecidable. 
Consider a consistent $U'\supseteq_{\sf c} U$.
Suppose $U'$ is decidable. Let $U^\ast$ be the theory given by the $U'$-theorems in the $U$-language.
Then, $U^\ast \sle U$ and $U^\ast$ is consistent and decidable. \emph{Quod impossibile}. 

\medskip
\emph{We prove \textup(b\textup).} The left-to-right direction is trivial. We treat right-to-left. Suppose $U$ is essentially hereditarily undecidable.
Suppose $U'\supseteq_{\sf c} U$ and  $W \sls U'$. Let ${U'}^\ast$ be the set of consequences of $U'$ in 
the $U$-language and let $W^\ast$ be the set of consequences of $W$ in the $U$-language.
Then $U \sls {U'}^{\ast}$ and $W^\ast \sls {U'}^{\ast}$. If $W$ would be decidable, then
so would be $W^\ast$. \emph{Quod impossibile}.  
\end{proof}

The next theorem combines Theorems~4 and 5 of \cite[Part I]{tars:unde53}. 

\begin{theorem}\label{snaaksesmurf}
Suppose $U$ is decidable and $\phi$ is a sentence in the $U^{(\vv c)}$-language.
Then $U^{(\vv c)}+\phi$ is decidable.
\end{theorem}

\begin{proof}
We have $U^{(\vv c)}+\phi\vdash \psi$ iff $U \vdash \forall \vv w\, (\phi[\vv c:\vv w] \to \psi[\vv c:\vv w])$.
\end{proof}

\subsection{c-Lax Interpretability}

We define  \emph{$V$ c-laxly interprets $U$}, or   $V \jump_{\sf c}U$, iff, for all consistent $V'\supseteq_{\sf c} V$, there is a consistent
$V''\supseteq_{\sf c} V'$, such that $V''\rhd U$.   

\begin{theorem}
$V \jump_{\sf c} U$ iff, for all $V'\sle V$, there is a $V'' \supseteq_{\sf c} V$, such that $V''\rhd U$.
\end{theorem}

\begin{proof}
Left-to-right is trivial. We treat right-to-left. We assume the right-hand-side of our equivalence.
Suppose $V'\supseteq_{\sf c} V$  and $V'$ is consistent. Let $\vv c$ be the relevant new constants. 
We define ${V'}^\ast$ as the theory given by the consequences of $V'$ in the $V$-language. 
Let $Z\supseteq_{\sf c} {V'}^\ast$ be a consistent theory, such that $Z \rhd U$.
Clearly, we may choose the new constants of $Z$, say $\vv d,$ distinct from $\vv c$.
We take $V'' := {V'}^{(\vv d)}\cup Z^{(\vv c)}$.
Clearly, $V'' \supseteq_{\sf c} V'$ and $V'' \rhd U$.

We claim that $V''$ is consistent. If not, for some $\psi(\vv c,\vv d)$, we would have
${V'}^{(\vv d)} \vdash \psi(\vv c,\vv d)$ and $Z^{(\vv c)} \vdash \neg\, \psi(\vv c,\vv d)$.
It follows that ${V'}^\ast \vdash \exists \vv v\, \forall \vv w\,\psi(\vv v,\vv w)$, So,
$Z \vdash  \exists \vv v\, \forall \vv w\,\psi(\vv v,\vv w)$. On the other hand, it follows
that $Z \vdash \forall \vv v\,   \neg\, \psi(\vv v,\vv d)$. This contradicts the fact that
$Z$ is consistent.
\end{proof}

We prove that essential hereditary undecidability is preserved over c-lax interpretability.

\begin{theorem}
Suppose $V \jump_{\sf c} U$ and $U$ is essentially hereditarily undecidable. Then $V$ is essentially hereditarily undecidable. 
\end{theorem}

\begin{proof}
Suppose $V \jump_{\sf c} U$ and $U$ is essentially hereditarily undecidable. 
Suppose $W$ is a decidable theory in the same language as $V$ that is  consistent with $V$. 
We have $ V\sls V\cup W =: V'$.
So, there is a consistent $V'' \supseteq_{\sf c} V'$ such that $V'' \rhd U$.   Say the witnessing constants are $\vec c$ and the witnessing
 translation is $\tau$. Let $Z := W^{(\vv c)} + \mf{id}^\tau$. Then, by Theorem~\ref{snaaksesmurf},  $Z$ decidable. 
Since $Z\sls V''$, 
  so $T:= \verz{\phi \mid Z \vdash \phi^\tau}$ is consistent with $U$.  Since $Z$ is decidable $T$ is decidable. \emph{Quod non}.
\end{proof}

\begin{theorem}
Suppose $V \rhd_{\sf pl} U$.
Then, $V \jump_{\sf c} U$.
\end{theorem}

\begin{proof}
Suppose $V \rhd_{\sf pl} U$.  Let $\tau$ be the witnessing translation.
Consider any consistent $V' \sle V$. 
We find that $V'$ proves $\exists \vv p\; \phi^{\tau,\vv p}$, for all
$\phi$ such that $U \vdash \phi$. We take a sequence of fresh constants 
$\vv c$ of the length of $\vv p$.  Then, by compactness,
$V'' :=V' + \verz{\phi^{\tau, \vv c} \mid U \vdash \phi}$ is consistent and, clearly,
$V'' \supseteq_{\sf c} V'$ and $V'' \rhd U$.
\end{proof}

\section{An Alternative Proof of Theorem~\ref{reptoei}}\label{rosser}
We show that if $U$ is $\Sigma^0_1$-representative, then it is effectively inseparable.
Let $\Phi$ witness that $U$ is $\Sigma^0_1$-representative. 
 Suppose $\mathcal X_0$ and $\mathcal X_1$ are disjoint  RE sets.
 Via a bijective G\"odel numbering, we can view these as
 sets of $U$-sentences. Suppose that
   $\downp{U} \subseteq \mc X_0$ and $\downr{U} \subseteq \mathcal X_1$.

Suppose $\exists u \, \delta_i(u,x)$ represents $\Phi(x)\in \mathcal X_i$ where $\delta_i$ is a pure $\Delta_0$-formula.

We need a special form of the G\"odel fixed point construction. We substitute a formula $\rho$ in  $\sigma_0(v) < \sigma_1(v)$ in the
following way. First we  substitute the numeral of the G\"odel number of $\rho$ in each of the $\sigma_i$.
Then, we rewrite each of the resulting formulas to pure 1-$\Sigma^0_1$-form and, subsequently, we apply witness comparison.
Suppose {\sf Sub} is this function.
Let ${\sf sub}(x,y,z)$ be a pure 1-$\Sigma^0_1$-representation of
the graph of {\sf Sub}.
Say ${\sf sub}(x,y,z)$ is $\exists w\, \nu(w,x,y,z)$, where $\nu$ is pure $\Delta_0$ and 
the witness $w$ will always majorise $x$, $y$ and $z$. 
We consider the formula $\phi(v)$:
\begin{multline*}
 (\exists a \, \exists w\bles a \, \exists z\bles a\, (\nu(w,v,v,z) \wedge \exists u\bles a\, \delta_1(u,z))) \; < \\
 (\exists b \, \exists w\bles a \, \exists z\bles b\, (\nu(w,v,v,z) \wedge \exists u\bles b\, \delta_0(u,z))).
\end{multline*}

\noindent
Let $\rho := {\sf Sub}(\phi,\phi)$. We note that $\rho$ is of the literal form $\rho_1 < \rho_0$, where the $\rho_i$ are
pure 1-$\Sigma^0_1$-sentences.

Suppose $\Phi(\rho) \in \mathcal X_i$. Then $\rho_i$. So either $\rho$ or $\rho^\bot$. 
If we have $\rho$, then $\Phi(\rho) \in \mathcal X_1$. Moreover, by $\Sigma 1$, we find $U\vdash \Phi(\rho)$, and, hence,
$\Phi(\rho)\in \mathcal X_0$. \emph{Quod impossibile}. If we have $\rho^\bot$, then
$\Phi(\rho) \in \mathcal X_0$. Moreover,
 by $\Sigma 3$, we have $U \vdash \neg\, \Phi(\rho)$ and, hence, $\Phi(\rho) \in \mathcal X_1$. 
 \emph{Quod impossibile iterum}.
 \emph{Ergo}, the formula $\Phi(\rho)$ is in none of the $\mc X_i$.
 
 \begin{remark}
 Inspection of the proof shows that in order to derive Rosser's theorem we do not need the formalisation of the Fixed Point Lemma  inside the given theory.
 \end{remark}
 
 \begin{remark}
 The above proof can be somewhat simplified by employing a self-referential G\"odel numbering. See, e.g., \cite{krip:gode23} or \cite{grab:self23}.
 However, even with this strategy, there will be some detail on how to handle substitution of numerals.
  \end{remark}
 \end{document}